\documentclass[11pt]{amsart}

\usepackage{enumerate, amsmath, amsthm, amsfonts, amssymb, xy,  mathrsfs, graphicx, paralist, lscape}
\usepackage[usenames, dvipsnames]{color}
\usepackage{fullpage}
\usepackage{blkarray}
\usepackage[normalem]{ulem}

\usepackage{tikz, tikz-cd}
\usetikzlibrary{matrix,arrows}
\usepackage{todonotes}
\numberwithin{equation}{section}
\newtheorem{theorem}[equation]{Theorem}

\newtheorem{proposition}[equation]{Proposition}
\newtheorem{lemma}[equation]{Lemma}
\newtheorem{corollary}[equation]{Corollary}

\theoremstyle{definition}
\newtheorem{rmk}[equation]{Remark}
\newenvironment{remark}[1][]{\begin{rmk}[#1] \pushQED{\qed}}{\popQED \end{rmk}}
\newtheorem{rmks}[equation]{Remarks}

\newtheorem{eg}[equation]{Example}
\newenvironment{example}[1][]{\begin{eg}[#1] \pushQED{\qed}}{\popQED \end{eg}}
\newtheorem{defn}[equation]{Definition}
\newenvironment{definition}[1][]{\begin{defn}[#1]\pushQED{\qed}}{\popQED \end{defn}}
\newtheorem{ques}[equation]{Question}

\newtheorem{notn}[equation]{Notation}

\usepackage[colorlinks=true, pdfstartview=FitV, linkcolor=blue, citecolor=blue, urlcolor=blue]{hyperref}
\newcommand{\fG}{\mathfrak{G}}

\newcommand{\fS}{\mathfrak{S}}

\newcommand{\ba}{\mathbf{a}}

\newcommand{\bb}{\mathbf{b}}

\newcommand{\fc}{\mathfrak{c}}

\newcommand{\bd}{\mathbf{d}}

\newcommand{\bs}{\mathbf{s}}

\newcommand{\bt}{\mathbf{t}}

\newcommand{\bv}{\mathbf{v}}

\newcommand{\bw}{\mathbf{w}}
\newcommand{\fw}{\mathfrak{w}}

\renewcommand{\phi}{\varphi}

\renewcommand{\tilde}[1]{\widetilde{#1}}

\newcommand{\comment}[1]{}

\makeatletter
\def\Ddots{\mathinner{\mkern1mu\raise\p@
\vbox{\kern7\p@\hbox{.}}\mkern2mu
\raise4\p@\hbox{.}\mkern2mu\raise7\p@\hbox{.}\mkern1mu}}
\makeatother

\DeclareMathOperator{\codim}{codim}

\DeclareMathOperator{\Mat}{Mat}

\newcommand{\GL}{\mathbf{GL}}

\usetikzlibrary{calc}
\newcommand{\latticepath}[4]{ % #1 = box diameter
  \coordinate (L) at #2;
  \foreach \x/\y/\a in {#4} {
    \coordinate (L1) at ($ (L) + ( #1 * \x , #1 * \y ) $);
    \draw[color=\a, #3] (L) -- (L1);
    \coordinate (L) at (L1);
  }
}

\newcommand{\pipedream}[5]{ % #1 = box diameter
  \coordinate (P) at ($ #2 - (#1, 0 ) $);

  \coordinate (Q) at ($ (P) + ( 1.5 * #1, 0.5 * #1 ) $);
  \foreach \object in {#3} {
    \node at (Q) {\object};
    \coordinate (Q) at ($ (Q) + ( #1, 0 ) $);
  }

  \coordinate (Q) at ($ (P) + ( 0.5*#1, -0.5*#1 ) $);
  \foreach \object in {#4} {
    \node at (Q) {\object};
    \coordinate (Q) at ($ (Q) + ( 0, -#1 ) $);
  }

  \foreach \x/\y/\z/\a/\b in {#5} {
    \coordinate (P1) at ($ (P) + ( #1 * \y , -#1 * \x ) + ( 0      , #1 / 2 ) $);
    \coordinate (P2) at ($ (P) + ( #1 * \y , -#1 * \x ) + ( #1     , #1 / 2 ) $);
    \coordinate (P3) at ($ (P) + ( #1 * \y , -#1 * \x ) + ( #1 / 2 , #1     ) $);
    \coordinate (P4) at ($ (P) + ( #1 * \y , -#1 * \x ) + ( #1 / 2 , 0      ) $);
    \coordinate (P5) at ($ (P) + ( #1 * \y , -#1 * \x ) + ( #1 / 2 , #1 / 2 ) $);
    
    \ifnum \z = 0
      \draw[rounded corners=4, color=\a, thick] (P1) -- (P5) -- (P3);
      \draw[rounded corners=4, color=\b, thick] (P4) -- (P5) -- (P2);
    \else
      \draw[rounded corners=0.2, color=\a, thick] (P3) -- (P5) -- (P4);
      \draw[rounded corners=0.2, color=\b, thick] (P1) -- (P5) -- (P2);
    \fi
  }
}

\def\textcross{
  \begin{minipage}{13pt}
    \begin{tikzpicture}[scale=1,>=latex]
      \pipedream{0.4}{(0,0)}{$$}{$$}{0/0/1/black/black}
    \end{tikzpicture}
  \end{minipage}
}

\def\textturn{
  \begin{minipage}{13pt}
    \begin{tikzpicture}[scale=1]
      \pipedream{0.4}{(0,0)}{$$}{$$}{0/0/0/black/black}
    \end{tikzpicture}
  \end{minipage}
}

\renewcommand{\GL}{{GL}}

\DeclareMathOperator{\Laces}{Laces}
\DeclareMathOperator{\Pipes}{Pipes}
\DeclareMathOperator{\PipeNet}{PipeNet}
\DeclareMathOperator{\RPipeNet}{RPipeNet}
\DeclareMathOperator{\RPipes}{RPipes}
\newcommand{\rep}{\texttt{rep}}

\makeatletter  
\def\bbordermatrix#1{\begingroup \m@th
  \@tempdima 4.75\p@
  \setbox\z@\vbox{%
    \def\cr{\crcr\noalign{\kern2\p@\global\let\cr\endline}}%
    \ialign{$##$\hfil\kern2\p@\kern\@tempdima&\thinspace\hfil$##$\hfil
      &&\quad\hfil$##$\hfil\crcr
      \omit\strut\hfil\crcr\noalign{\kern-\baselineskip}%
      #1\crcr\omit\strut\cr}}%
  \setbox\tw@\vbox{\unvcopy\z@\global\setbox\@ne\lastbox}%
  \setbox\tw@\hbox{\unhbox\@ne\unskip\global\setbox\@ne\lastbox}%
  \setbox\tw@\hbox{$\kern\wd\@ne\kern-\@tempdima\left[\kern-\wd\@ne
    \global\setbox\@ne\vbox{\box\@ne\kern2\p@}%
    \vcenter{\kern-\ht\@ne\unvbox\z@\kern-\baselineskip}\,\right]$}%
  \null\;\vbox{\kern\ht\@ne\box\tw@}\endgroup}
\makeatother

\DeclareMathOperator{\rot}{rot}

\title{Combinatorial proofs of the \\type $A$ quiver component formulas}

\author{Aidan Lindberg}
\author{Jenna Rajchgot}
\address{University of Toronto, Department of Mathematics, Toronto, ON, Canada}
\email[Aidan Lindberg]{aidan.lindberg@mail.utoronto.ca}
\address{McMaster University, Department of Mathematics and Statistics, Hamilton, ON, Canada}
\email[Jenna Rajchgot]{rajchgoj@mcmaster.ca}
\begin{document}
\maketitle

\begin{abstract} 

 The $K$-theoretic quiver component formula expresses the $K$-polynomial of a type $A$ quiver locus as an alternating sum of products of double Grothendieck polynomials. This formula was conjectured by A. Buch and R. Rim\'anyi and later proved by R. Kinser, A. Knutson, and the second author. 
 We provide a new proof of this formula which replaces Gr\"obner degenerations by combinatorics. 
 Along the way, we obtain a new proof of A. Buch and R. Rim\'anyi's cohomological quiver component formula. Again, our proof replaces geometric techniques by combinatorics. 
\end{abstract}

\setcounter{tocdepth}{1}
\tableofcontents

\section{Introduction}

Given a quiver $Q$ with vertex set $Q_0$, arrow set $Q_1$, and dimension vector $\mathbf{d}:Q_0 \to \mathbb{Z}_{\geq 0}$, one forms the representation space 
\begin{align*}
\rep_Q(\mathbf{d})=\prod_{a \in Q_1} \Mat(\mathbf{d}(ta), \mathbf{d}(ha))\,,
\end{align*}
where $ta\in Q_0$ and $ha\in Q_0$ are the tail and head vertices of the arrow $a\in Q_1$, and $\text{Mat}(m,n)$ is the space of $m\times n$ matrices with entries in a field $\mathbb{K}$. The product of general linear groups $\GL(\mathbf{d})=\prod_{z \in Q_0} \GL_{\mathbf{d}(z)}(\mathbb{K})$ acts on $\rep_Q(\bd)$ on the right by 
\[(V_a)_{a \in Q_1}\bullet (g_z)_{z \in Q_0} = (g_{ta}^{-1}V_a g_{ha})_{a \in Q_1},\] where  $(V_a)_{a \in Q_1} \in \rep_Q(\mathbf{d})$, and $(g_z)_{z \in Q_0} \in \GL(\mathbf{d})$. The closure of a $\GL(\mathbf{d})$ orbit $\Omega^\circ$ in $\rep_Q(\mathbf{d})$ is called a \textit{quiver locus}, and will be denoted by $\Omega$. 
 
 The restriction of the $GL(\bd)$ action to the torus $T\leq GL(\bd)$ of tuples of diagonal matrices induces a $\mathbb{Z}^d$-grading, where $d = \sum_{z\in Q_0} \bd(z)$, on the coordinate ring of each quiver locus. In this article, we study known formulas for multidegrees and $K$-polynomials, with respect to this grading, of quiver loci in representation spaces of  type $A$ quivers. We focus primarily on the \emph{component formulas}: the cohomological component formula (see Theorem \ref{thm:3.9}) expresses the multidegree of a type $A$ quiver locus as a positive sum of products of double Schubert polynomials, and the $K$-theoretic component formula (see Theorem \ref{thm:4.16}) expresses the $K$-polynomial as an alternating sum of products of double Grothendieck polynomials. In this paper we give new proofs of these formulas which are combinatorial in nature. To the best of our knowledge, all prior proofs of these formulas were geometric; our proofs replace these geometric arguments with the combinatorics of \emph{lacing diagrams} and \emph{pipe dreams}. 
 
  Our first main result, Theorem \ref{thm:3.9}, is a new proof of the cohomological component formula for arbitrarily oriented type $A$ quiver loci. Component formulas for \emph{equioriented} type $A$ quivers (i.e., type $A$ quivers where all arrows point in the same direction) were proven in \cite{KMS, yongComponent, Miller, Buch, BFR}. In the case of arbitrary orientation, the cohomological component formula was first proven by Buch and Rim\'anyi in \cite{BuchRim} using the geometry and representation theory of the quiver locus. 
  As a corollary, Buch and Rim\'anyi provide a formula for the codimension of a type $A$ quiver locus $\Omega$, in the representation space $\rep_Q(\bd)$, in terms of the number of crossings in the extension of an associated \textit{minimal lacing diagram} \cite[Corollary 2]{BuchRim}. 
  In the present paper, we give a new combinatorial proof of this codimension formula in the bipartite type $A$ quiver setting (see Corollary \ref{cor:codim}) which makes use of the characterization of minimal lacing diagrams given in \cite{BuchRim}. We then use this formula in our proof of the cohomological quiver component formula. We note that the  codimension formula is also used in an essential way in \cite{KKR} to prove the $K$-theoretic component formula.

 Our second main result, Theorem \ref{thm:4.16}, is a new proof of the $K$-theoretic component formula for type $A$ quiver loci. This formula was conjectured by Buch and Rim\'anyi in \cite{BuchRim} and later proved by Kinser, Knutson, and the second author in \cite{KKR}. The proof of the $K$-theoretic component formula in \cite{KKR} relies on a Gr\"obner degeneration, as well as a M\"obius function computation (see \cite{KKR} or \cite{KinserICRA} for details). 
 We replace these Gr\"obner degeneration and M\"obius function computations by combinatorics of lacing diagrams and pipe dreams. Our proof is essentially a \emph{bipartite} (i.e. source sink) orientation version of the proof of the equioriented $K$-theoretic quiver component formula given by Buch in \cite[\S6]{MR2114821}; indeed, we adapt various results for equioriented quivers from there, as well as from \cite{MR2137947, BFR, yongComponent}, to bipartite orientation. We note that a couple of these adapted results were already used in \cite[\S4.6]{KKR}, with details left to the reader. For further history and motivation on these and related formulas, see \cite{KinserICRA} or \cite{KKR} and references therein. 
 
 In the course of proving the component formulas, we also give new combinatorial proofs of the \emph{pipe formulas} from \cite{KKR} (see Section \ref{sect:codim}). It is worth noting that while our proofs of the cohomological and $K$-theoretic component formulas both rely on the combinatorics of lacing diagrams and pipe dreams, the combinatorial results which lead to these formulas are rather different in each case. Indeed, the cohomological formula follows from direct analysis of reduced pipe dreams for the Zelevinsky permutation $v(\Omega)$ of the quiver locus $\Omega$, while the $K$-theoretic formula is proven by studying sequences of permutations obtained from (possibly non-reduced) pipe dreams for $v(\Omega)$.
   
 Throughout the paper we work only with bipartite type $A$ quivers. This is because the arbitrarily oriented type $A$ quiver versions of the formulas we consider follow from the bipartite versions (see \cite[\S5]{KKR} for details). We assume the leftmost and rightmost vertices of our bipartite type $A$ quivers are sources. We label source vertices by $(y_i)_{i=0}^n$, sink vertices by $(x_i)_{i=1}^n$, right-pointing arrows by $(\beta_i)_{i=1}^n$, and left-pointing arrows by $(\alpha_i)_{i=1}^n$. Subscripts increase when reading from right to left.
For example,
\begin{equation}\label{eq:quiverEg}
\begin{tikzpicture}
\node (y2) at (0,1) {$y_2$};
\node (x2) at (1,0) {$x_2$};
\node (y1) at (2,1) {$y_1$};
\node (x1) at (3,0) {$x_1$};
\node (y0) at (4,1) {$y_0$};

\node (b2) at (.3, .3) {$\beta_2$};
\node (a2) at (1.7, .3) {$\alpha_2$};
\node (b1) at (2.3, .3) {$\beta_1$};
\node (a1) at (3.7, .3) {$\alpha_1$};

\draw[->] (y2) -- (x2);
\draw[->] (y1) -- (x2);
\draw[->] (y1) -- (x1);
\draw[->] (y0) -- (x1);
\end{tikzpicture}
\end{equation}
is a bipartite type $A_5$ quiver with our labeling. 

To each sink vertex $x_k$, we associate a list of variables $\mathbf{s}^k=s_1^k, \dots, s_{\mathbf{d}(x_k)}^k$, and to each source vertex $y_k$, we associate a list $\mathbf{t}^k=t_1^k, \dots, t_{\mathbf{d}(y_k)}^k$. 
We concatenate these lists to $\mathbf{s}=\mathbf{s}^n, \dots, \mathbf{s}^1$ and $\mathbf{t}=\mathbf{t}^0,\dots, \mathbf{t}^n$. By \cite[Theorem 3.2, Theorem 5.9]{KKR}, the $K$-polynomial of a bipartite type $A$ quiver locus $\Omega\subseteq \rep_Q(\bd)$ is given by the quotient of specialized double Grothendieck polynomials on the left below and the multidegree is given by the quotient of specialized double Schubert polynomials on the right below: 
\begin{equation}\label{eq:ratioformulas}
 \mathcal{KQ}_{\Omega}(\mathbf{t}/\mathbf{s})=\frac{\mathfrak{G}_{v(\Omega)}(\mathbf{t},\mathbf{s};\mathbf{s},\mathbf{t})}{\mathfrak{G}_{v_*}(\mathbf{t},\mathbf{s};\mathbf{s},\mathbf{t})}, \qquad 
\mathcal{Q}_{\Omega}(\mathbf{t}-\mathbf{s})=\frac{\mathfrak{S}_{v(\Omega)}(\mathbf{t},\mathbf{s};\mathbf{s},\mathbf{t})}{\mathfrak{S}_{v_*}(\mathbf{t},\mathbf{s};\mathbf{s},\mathbf{t})}\,.
\end{equation}
Here, $v(\Omega)$ and $v_*$ are particular permutations associated to $\Omega$ and $\rep_Q(\bd)$ respectively (see Section \ref{sec:2.3}). These \emph{bipartite ratio formulas} were proved in \cite{KKR}, and are bipartite versions of the equioriented ratio formulas from \cite{KMS}. They follow easily from results in the Schubert variety literature, together with the \emph{bipartite Zelevinsky map} which identifies a bipartite type-$A$ quiver locus with the intersection of a Schubert variety and an opposite Schubert cell (see \cite{KR}). \emph{We take the bipartite ratio formulas as our starting point in this paper;} we derive the bipartite pipe and component formulas from the bipartite ratio formulas using combinatorial arguments.

\subsection*{Acknowledgments} Some of the results in this paper were proven during an undergraduate summer research project at the University of Saskatchewan. Both authors were partially supported by NSERC grant 2017-05732. J.R. was also partially supported by NSERC grant 2023-04800. A.L. was partially supported by an NSERC CGS-D scholarship, as well as a University of Toronto FAST fellowship. We thank Ryan Kinser for helpful correspondence, especially for suggesting the idea for this project.

%~~~~~~~~~~~~~~~~~~~~~~~~~~~~~~~~~~~~~~~~~~~~~~~~~~~~~~~~~~~~

\section{Preliminaries}

In this section, we fix our conventions and briefly recall relevant combinatorial background. For more details on this background material, see \cite[\S2]{KKR}. 

\subsection{Permutations and partial permutations}\label{sec:Permutations}
 
We let $S_n$ denote the symmetric group on $n$ elements, and let $\tau_i\in S_n$ denote the simple transposition that interchanges $i$ and $i+1$. We write $\ell(v)$ for the length of a permutation $v\in S_n$. We represent $v\in S_n$ as an $n \times n$ permutation matrix which has a $1$ in entry $(i,j)$ if and only if $v(i)=j$, and zeros elsewhere.

A $k \times l$ matrix $w$ is a \emph{partial permutation matrix} if every entry of $w$ is either $0$ or $1$, and each row and each column has at most one $1$. Such a $w$ has a unique \textit{completion}, $c(w)$, defined to be the permutation matrix of smallest possible dimensions and minimal length such that $w$ lies in its northwest corner.

\begin{remark}\label{rem:completions}
If $w$ is a $k \times \ell$ partial permutation matrix, then we will often view its completion $c(w)$ as an element of $S_{k+\ell}$ as follows. If $c(w)$ is an $m \times m$ permutation matrix then for every $i$ such that $m < i \leq k+\ell$ we define $c(w)(i)=i$. 
\end{remark}

\subsection{Lacing diagrams}\label{sect:lacingDiagrams}

A \textit{lacing diagram} for a type $A$ quiver $Q$ with dimension vector $\mathbf{d}$ is a sequence of columns of dots, with arrows connecting dots in consecutive columns. The columns of dots are indexed by $Q_0$ in the same order that the elements of $Q_0$ appear in $Q$. There are $\mathbf{d}(z)$ dots in column $z\in Q_0$. If $a\in Q_1$ is a right-pointing (respectively left-pointing) arrow, align the columns of dots associated to $ta$ and $ha$ from their top dot (respectively bottom dot). There is at most one arrow entering or exiting the left (respectively right) of any dot. The direction of an arrow between columns $ta, ha$ in the lacing diagram agrees with the direction of $a\in Q_1$. The connected components of a lacing diagram are called its \textit{laces}. See Figure \ref{fig:laces}.

There is a one-to-one correspondence between lacing diagrams for $Q$ with dimension vector $\bd$ and representations $\bw = (w_a)_{a\in Q_1}$ in $\rep_Q(\bd)$, where $w_a$ is a $\bd(ta)\times \bd(ha)$ partial permutation matrix. Indeed, $w_a$ has a $1$ in position $(i,j)$ if and only if the associated lacing diagram has an arrow from the $i^{\rm{th}}$ dot from the top in column $ta$ to the $j^{\rm{th}}$ dot from the top in column $ha$. Because of this correspondence, we use the term lacing diagram to refer to both a sequence of partial permutations $\mathbf{w}\in \rep_Q(\bd)$ and its corresponding picture with dots and arrows.  

Given a lacing diagram $\bw = (w_a)_{a\in Q_1}$, we define its \emph{extended lacing diagram} to be the sequence of permutation matrices $\bv = (v_a)_{a\in Q_1}$, where $v_a = c(w_a)$ when $a$ is a right pointing arrow, and $v_a = \rot(c(\rot(w_a)))$ when $a$ is a left pointing arrow. Here, $\rot(w)$ maps the matrix $w$ to its 180 degree rotation. When we draw extended lacing diagrams as pictures with dots and arrows, we use red for those \emph{virtual} dots and arrows which are in the extended lacing diagram, but not the original lacing diagram.

\begin{example}
The sequence of partial permutations associated to the lacing diagram on the left of Figure \ref{fig:laces} is given by 

\begin{equation}
\bw = \left( \begin{bmatrix}
1 & 0 & 0 \\
0 & 1 & 0 \\
\end{bmatrix},
\begin{bmatrix}
1 & 0 & 0 \\ 
0 & 0 & 1 \\
0 & 1 & 0 \\
\end{bmatrix},
\begin{bmatrix}
1 & 0 \\
0 & 1 \\
0 & 0 \\
\end{bmatrix},
\begin{bmatrix}
1 & 0 
\end{bmatrix} \right).
\end{equation}
The associated extended lace diagram appears on right in Figure \ref{fig:laces} and corresponds to the sequence of permutations
\begin{align}
\left( \begin{bmatrix}
1 & 0 & 0 \\
0 & 1 & 0 \\
0 & 0 & 1 \\
\end{bmatrix},
\begin{bmatrix}
1 & 0 & 0 \\
0 & 0 & 1 \\
0 & 1 & 0 \\
\end{bmatrix},
\begin{bmatrix}
1 & 0 & 0 \\
0 & 1 & 0 \\
0 & 0 & 1 \\
\end{bmatrix},
\begin{bmatrix}
0 & 1 \\
1 & 0 \\
\end{bmatrix} \right)\,.
\end{align}
\end{example}

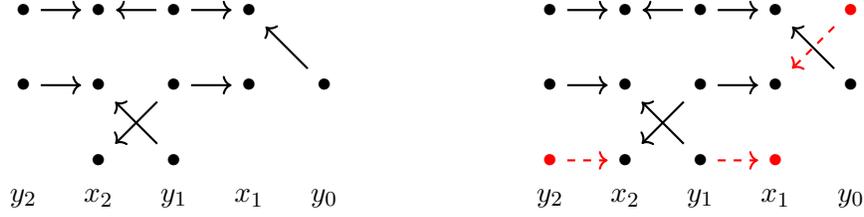
\begin{figure}
\begin{center}
\begin{tikzpicture}
\node (w11) at (0,2) {$\bullet$};
\node (w12) at (0,1) {$\bullet$};

\node (w21) at (1,2) {$\bullet$};
\node (w22) at (1,1) {$\bullet$};
\node (w23) at (1,0) {$\bullet$};

\node (w31) at (2,2) {$\bullet$};
\node (w32) at (2,1) {$\bullet$};
\node (w33) at (2,0) {$\bullet$};

\node (w41) at (3,2) {$\bullet$};
\node (w42) at (3,1) {$\bullet$};

\node (w51) at (4,1) {$\bullet$};

\node (y2) at (0,-.5) {$y_2$};
\node (x2) at (1,-.5) {$x_2$};
\node (y1) at (2,-.5) {$y_1$};
\node (x1) at (3,-.5) {$x_1$};
\node (y0) at (4,-.5) {$y_0$};

\draw[->, thick] (w11) -- (w21);
\draw[->, thick] (w12) -- (w22);

\draw[->, thick] (w31) -- (w21);
\draw[->, thick] (w32) -- (w23);
\draw[->, thick] (w33) -- (w22);

\draw[->, thick] (w31) -- (w41);
\draw[->, thick] (w32) -- (w42);

\draw[->, thick] (w51) -- (w41);

\node (cw11) at (7,2) {$\bullet$};
\node (cw12) at (7,1) {$\bullet$};
\node[red] (cw13) at (7,0) {$\bullet$};

\node (cw21) at (8,2) {$\bullet$};
\node (cw22) at (8,1) {$\bullet$};
\node (cw23) at (8,0) {$\bullet$};

\node (cw31) at (9,2) {$\bullet$};
\node (cw32) at (9,1) {$\bullet$};
\node (cw33) at (9,0) {$\bullet$};

\node (cw41) at (10,2) {$\bullet$};
\node (cw42) at (10,1) {$\bullet$};
\node[red] (cw43) at (10,0) {$\bullet$};

\node (cw51) at (11,1) {$\bullet$};
\node[red] (cw50) at (11,2) {$\bullet$};

\node (cy2) at (7,-.5) {$y_2$};
\node (cx2) at (8,-.5) {$x_2$};
\node (cy1) at (9,-.5) {$y_1$};
\node (cx1) at (10,-.5) {$x_1$};
\node (cy0) at (11,-.5) {$y_0$};

\draw[->, thick] (cw11) -- (cw21);
\draw[->, thick] (cw12) -- (cw22);
\draw[->, thick, dashed, red] (cw13) -- (cw23);

\draw[->, thick] (cw31) -- (cw21);
\draw[->, thick] (cw32) -- (cw23);
\draw[->, thick] (cw33) -- (cw22);

\draw[->, thick] (cw31) -- (cw41);
\draw[->, thick] (cw32) -- (cw42);
\draw[->, thick, dashed, red] (cw33) -- (cw43);

\draw[->, thick] (cw51) -- (cw41);
\draw[->, thick, dashed, red] (cw50) -- (cw42);

\end{tikzpicture}
\end{center}
\caption{A lacing diagram for quiver (1.1) and its associated extended lacing diagram.}
\label{fig:laces}
\end{figure}

Two lacing diagrams are in the same $\GL(\mathbf{d})$ orbit if and only if they have the same number of laces between any two pairs of columns \cite[Lemma 3.2]{KMS}. The length of a lacing diagram $\mathbf{w}$, denoted $|\bf{w}|$, is the number of crossings in its associated extended lacing diagram. 

Let $\Omega^\circ$ denote the open orbit of a type $A$ quiver locus $\Omega$. A lacing diagram $\mathbf{w}\in \Omega^\circ$ is called \textit{minimal} if its length is minimal among all lacing diagrams in $\Omega^\circ$. For example, the lacing diagram in Figure \ref{fig:laces} is minimal. Denote the set of minimal lacing diagrams in $\Omega^\circ$ by $W(\Omega)$. 
By \cite[Proposition 1]{BuchRim}, a lacing diagram $\bw \in W(\Omega)$ can be obtained from any other $\bw' \in W(\Omega)$ through a series of \textit{reduced lacing diagram moves} of the form 
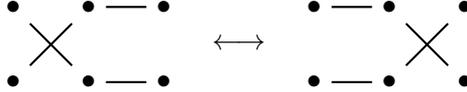
\begin{center}
\begin{tikzpicture}
\node (1) at (-3,1) {$\bullet$}; 
\node (2) at (-3,0) {$\bullet$};
\node (3) at (-2,1) {$\bullet$};
\node (4) at (-2,0) {$\bullet$};
\node (5) at (-1,1) {$\bullet$};
\node (6) at (-1,0) {$\bullet$};

\node (arrow) at (0,.5) {$\longleftrightarrow$};

\node (66) at (3,1) {$\bullet$}; 
\node (55) at (3,0) {$\bullet$};
\node (44) at (2,1) {$\bullet$};
\node (33) at (2,0) {$\bullet$};
\node (22) at (1,1) {$\bullet$};
\node (11) at (1,0) {$\bullet$};

\draw[-, thick] (1) -- (4);
\draw[-, thick] (2) -- (3);
\draw[-, thick] (3) -- (5);
\draw[-, thick] (4) -- (6);

\draw[-, thick] (11) -- (33);
\draw[-, thick] (22) -- (44);
\draw[-, thick] (33) -- (66);
\draw[-, thick] (44) -- (55);
\end{tikzpicture}
\end{center}
where both middle dots and at least one dot in each outer column is non-virtual. 
Define the set of $K$-theoretic lacing diagrams for $\Omega$, denoted $KW(\Omega)$, to be those lacing diagrams which can be obtained from a minimal lacing diagram using \textit{$K$-theoretic lacing diagram moves} of the form
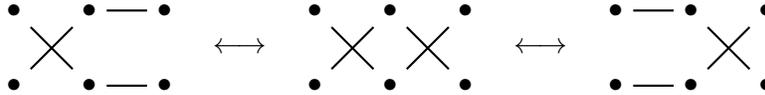
\begin{center}
\begin{tikzpicture}
\node (1) at (-3,1) {$\bullet$}; 
\node (2) at (-3,0) {$\bullet$};
\node (3) at (-2,1) {$\bullet$};
\node (4) at (-2,0) {$\bullet$};
\node (5) at (-1,1) {$\bullet$};
\node (6) at (-1,0) {$\bullet$};

\node (arrow) at (0,.5) {$\longleftrightarrow$};

\node (66) at (3,1) {$\bullet$}; 
\node (55) at (3,0) {$\bullet$};
\node (44) at (2,1) {$\bullet$};
\node (33) at (2,0) {$\bullet$};
\node (22) at (1,1) {$\bullet$};
\node (11) at (1,0) {$\bullet$};

\node (arrow2) at (4, .5) {$\longleftrightarrow$};

\node (666) at (7,1) {$\bullet$}; 
\node (555) at (7,0) {$\bullet$};
\node (444) at (6,1) {$\bullet$};
\node (333) at (6,0) {$\bullet$};
\node (222) at (5,1) {$\bullet$};
\node (111) at (5,0) {$\bullet$};

\draw[-, thick] (1) -- (4);
\draw[-, thick] (2) -- (3);
\draw[-, thick] (3) -- (5);
\draw[-, thick] (4) -- (6);

\draw[-, thick] (11) -- (44);
\draw[-, thick] (22) -- (33);
\draw[-, thick] (33) -- (66);
\draw[-, thick] (44) -- (55);

\draw[-, thick] (111) -- (333);
\draw[-, thick] (222) -- (444);
\draw[-, thick] (333) -- (666);
\draw[-, thick] (444) -- (555);
\end{tikzpicture}
\end{center}
where both of the dots in the middle column are non-virtual and are adjacent in their column, and at least one dot in each of the outer columns is non-virtual. We would like to emphasize that lacing diagram moves are always applied to the completed diagrams. 

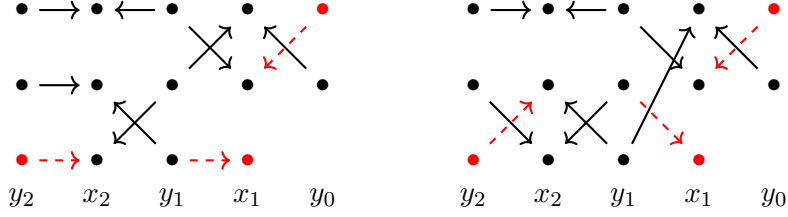
\begin{figure}
\begin{center}
\begin{tikzpicture}
\node (w11) at (0,2) {$\bullet$};
\node (w12) at (0,1) {$\bullet$};
\node[red] (w13) at (0,0) {$\bullet$};

\node (w21) at (1,2) {$\bullet$};
\node (w22) at (1,1) {$\bullet$};
\node (w23) at (1,0) {$\bullet$};

\node (w31) at (2,2) {$\bullet$};
\node (w32) at (2,1) {$\bullet$};
\node (w33) at (2,0) {$\bullet$};

\node (w41) at (3,2) {$\bullet$};
\node (w42) at (3,1) {$\bullet$};
\node[red] (w43) at (3,0) {$\bullet$};

\node (w51) at (4,1) {$\bullet$};
\node[red] (w50) at (4,2) {$\bullet$};

\node (y2) at (0,-.5) {$y_2$};
\node (x2) at (1,-.5) {$x_2$};
\node (y1) at (2,-.5) {$y_1$};
\node (x1) at (3,-.5) {$x_1$};
\node (y0) at (4,-.5) {$y_0$};

\draw[->, thick] (w11) -- (w21);
\draw[->, thick] (w12) -- (w22);
\draw[->, thick, dashed, red] (w13) -- (w23);

\draw[->, thick] (w31) -- (w21);
\draw[->, thick] (w32) -- (w23);
\draw[->, thick] (w33) -- (w22);

\draw[->, thick] (w31) -- (w42);
\draw[->, thick] (w32) -- (w41);
\draw[->, thick, dashed, red] (w33) -- (w43);

\draw[->, thick] (w51) -- (w41);
\draw[->, thick, dashed, red] (w50) -- (w42);

\node (cw11) at (6,2) {$\bullet$};
\node (cw12) at (6,1) {$\bullet$};
\node[red] (cw13) at (6,0) {$\bullet$};

\node (cw21) at (7,2) {$\bullet$};
\node (cw22) at (7,1) {$\bullet$};
\node (cw23) at (7,0) {$\bullet$};

\node (cw31) at (8,2) {$\bullet$};
\node (cw32) at (8,1) {$\bullet$};
\node (cw33) at (8,0) {$\bullet$};

\node (cw41) at (9,2) {$\bullet$};
\node (cw42) at (9,1) {$\bullet$};
\node[red] (cw43) at (9,0) {$\bullet$};

\node (cw51) at (10,1) {$\bullet$};
\node[red] (cw50) at (10,2) {$\bullet$};

\node (cy2) at (6,-.5) {$y_2$};
\node (cx2) at (7,-.5) {$x_2$};
\node (cy1) at (8,-.5) {$y_1$};
\node (cx1) at (9,-.5) {$x_1$};
\node (cy0) at (10,-.5) {$y_0$};

\draw[->, thick] (cw11) -- (cw21);
\draw[->, thick] (cw12) -- (cw23);
\draw[->, thick, dashed, red] (cw13) -- (cw22);

\draw[->, thick] (cw31) -- (cw21);
\draw[->, thick] (cw32) -- (cw23);
\draw[->, thick] (cw33) -- (cw22);

\draw[->, thick] (cw31) -- (cw42);
\draw[->, thick, dashed, red] (cw32) -- (cw43);
\draw[->, thick] (cw33) -- (cw41);

\draw[->, thick] (cw51) -- (cw41);
\draw[->, thick, dashed, red] (cw50) -- (cw42);
\end{tikzpicture}
\end{center}
\caption{Two $K$-theoretic lacing diagrams obtained from the minimal lacing diagram from Figure \ref{fig:laces}.}
\label{fig:KTheoryLaces}
\end{figure}

\begin{example}\label{ex:KThDiagrams}
    Starting with the extended lacing diagram in Figure \ref{fig:laces}, we can obtain the rightmost $K$-theoretic lacing diagram in Figure \ref{fig:KTheoryLaces} through the following two intermediate lacing diagrams: 
    
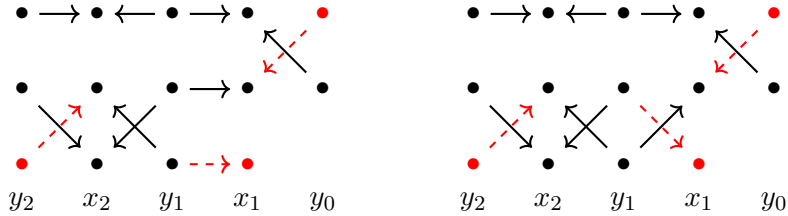
\begin{center}
        \begin{tikzpicture}
\node (w11) at (0,2) {$\bullet$};
\node (w12) at (0,1) {$\bullet$};
\node[red] (w13) at (0,0) {$\bullet$};

\node (w21) at (1,2) {$\bullet$};
\node (w22) at (1,1) {$\bullet$};
\node (w23) at (1,0) {$\bullet$};

\node (w31) at (2,2) {$\bullet$};
\node (w32) at (2,1) {$\bullet$};
\node (w33) at (2,0) {$\bullet$};

\node (w41) at (3,2) {$\bullet$};
\node (w42) at (3,1) {$\bullet$};
\node[red] (w43) at (3,0) {$\bullet$};

\node (w51) at (4,1) {$\bullet$};
\node[red] (w52) at (4,2) {$\bullet$};

\node (y2) at (0,-.5) {$y_2$};
\node (x2) at (1,-.5) {$x_2$};
\node (y1) at (2,-.5) {$y_1$};
\node (x1) at (3,-.5) {$x_1$};
\node (y0) at (4,-.5) {$y_0$};

\draw[->, thick] (w11) -- (w21);
\draw[->, thick] (w12) -- (w23);
\draw[->, thick, dashed, red] (w13) -- (w22);

\draw[->, thick] (w31) -- (w21);
\draw[->, thick] (w32) -- (w23);
\draw[->, thick] (w33) -- (w22);

\draw[->, thick] (w31) -- (w41);
\draw[->, thick] (w32) -- (w42);
\draw[->, thick, dashed, red] (w33) -- (w43);

\draw[->, thick] (w51) -- (w41);
\draw[->, thick, dashed, red] (w52) -- (w42);

\node (cw11) at (6,2) {$\bullet$};
\node (cw12) at (6,1) {$\bullet$};
\node[red] (cw13) at (6,0) {$\bullet$};

\node (cw21) at (7,2) {$\bullet$};
\node (cw22) at (7,1) {$\bullet$};
\node (cw23) at (7,0) {$\bullet$};

\node (cw31) at (8,2) {$\bullet$};
\node (cw32) at (8,1) {$\bullet$};
\node (cw33) at (8,0) {$\bullet$};

\node (cw41) at (9,2) {$\bullet$};
\node (cw42) at (9,1) {$\bullet$};
\node[red] (cw43) at (9,0) {$\bullet$};

\node (cw51) at (10,1) {$\bullet$};
\node[red] (cw50) at (10,2) {$\bullet$};

\node (cy2) at (6,-.5) {$y_2$};
\node (cx2) at (7,-.5) {$x_2$};
\node (cy1) at (8,-.5) {$y_1$};
\node (cx1) at (9,-.5) {$x_1$};
\node (cy0) at (10,-.5) {$y_0$};

\draw[->, thick] (cw11) -- (cw21);
\draw[->, thick] (cw12) -- (cw23);
\draw[->, thick, dashed, red] (cw13) -- (cw22);

\draw[->, thick] (cw31) -- (cw21);
\draw[->, thick] (cw32) -- (cw23);
\draw[->, thick] (cw33) -- (cw22);

\draw[->, thick] (cw31) -- (cw41);
\draw[->, thick, dashed, red] (cw32) -- (cw43);
\draw[->, thick] (cw33) -- (cw42);

\draw[->, thick] (cw51) -- (cw41);
\draw[->, thick, dashed, red] (cw50) -- (cw42);
        \end{tikzpicture}
    \end{center}
\end{example}

\subsection{Block structures on $d\times d$ matrices and the bipartite Zelevinsky permutation}\label{sect:Zperm} \label{sec:2.3}

We now recall the \emph{bipartite Zelevinsky permutation}, first introduced in \cite{KR}. Let $\bd$ be a dimension vector for a bipartite type $A$ quiver $Q$. Let $d_x=\sum_{j=1}^n \mathbf{d}(x_j)$, $d_y=\sum_{j=0}^n \mathbf{d}(y_j)$, and $d = d_x+d_y$. Throughout the paper, any $d\times d$ matrix (or grid) is given a \textit{block structure} by separating the rows into blocks labeled $y_0,\dots, y_n, x_n,\dots, x_1$ reading from top to bottom, and separating the columns into blocks labeled $x_n,\dots, x_1,y_0,\dots, y_n$ reading from left to right. A block row or column labeled by $v\in Q_0$ has size $\bd(v)$.   
We label certain blocks in the northwest $d_y\times d_x$ submatrix by arrows in $Q$: the block at the intersection of block row $y_{i-1}$ and block column $x_i$ will be called \textit{block} $\alpha_i$, and the block at the intersection of block row $y_{i}$ and block column $x_i$ will be called \textit{block} $\beta_i$. For example, in Figure \ref{fig:Zperm}, the $3\times 3$ identity matrix in block row $y_1$ and block column $x_2$ is block $\alpha_2$.  Define the \textit{snake region} of $M$ to be all $(i,j)$ such that $(i,j)$ is in some $\alpha_k$ block or some $\beta_k$ block.

Let $\Omega\subseteq \rep_Q(\bd)$ be a bipartite type $A$ quiver locus, let $\Omega^\circ$ be its open orbit, and let $\bw\in \Omega^\circ$ be a lacing diagram. The \textit{bipartite Zelevinsky permutation}  $v(\Omega)\in S_d$ is the unique permutation satisfying the following three conditions (see  \cite[Proposition 2.13]{KKR}): 
\begin{itemize}
\item[(\textit{Z1})] for $z_i, z_j \in Q_0$ with $z_i$ to the left of $z_j$ in $Q$ (including the case $z_i=z_j$), the number of $1$s in block $(z_i, z_j)$ is the number of laces in $\bw$ with left endpoint $z_i$ and right endpoint $z_j$;
\item[(\textit{Z2})] for $z_i, z_j \in Q_0$ with $z_i$ exactly one vertex to the right of $z_j$ in $Q$, the number of $1$s in block $(z_i, z_j)$ is the number of arrows between the columns of $\mathbf{w}$ indexed by $z_i$ and $z_j$;
\item[(\textit{Z3})] the $1$s are arranged from northwest to southeast in each block row and column. 
\end{itemize}

Note that for any $z_i, z_j \in Q_0$, the number of laces starting at $z_i$ and ending at $z_j$ is the same for every $\bw \in \Omega^{\circ}$. It follows that $v(\Omega)$ is independent of the choice of lacing diagram $\bw$ that appears in the defining properties $(Z1)$ and $(Z2)$. The Zelevinsky permutation associated to the lacing diagram on the left of Figure \ref{fig:laces}  is depicted in Figure \ref{fig:Zperm}. We denote by $v_\ast$ the Zelevinsky permutation associated to the representation space $\rep_{Q}(\bd)$.

\begin{figure}
\small{
\begin{align*}
\begin{blockarray}{cccc|cc|c|ccc|cc}
            &x_2 & &  & x_1 & & y_0 & y_1 &  &  & y_2 & & \\
      \begin{block}{c[ccccc|cccccc]}
y_0 &  &  &  & 1 & 0 &  & \\
\cline{1-1}
y_1 & 1 & 0 & 0 &  &  & &   \\
& 0 & 1 & 0 &  &  & & &  \\
& 0 & 0 & 1 &  &  & & & & \\
\cline{1-1}
y_2 &  &  &  & & & 1 & 0 & 0 & 0 & \\
 &  &  &  & & & 0 & 1 & 0 & 0 & &  \\
\cline{1-12}
x_2 &  &  &  &  0 & 1 &  \\
 &  &  &  &  0 & 0 &  &  &  &  &  1  &  0  & \\
 &  &  &  &  0 & 0 &  &  &  &  &  0  &  1  &\\
\cline{1-1}
x_1 & & & & & & & 0 & 1 & 0 & \\
 & & & & & & & 0 & 0 & 1 &\\
      \end{block}
   \end{blockarray}
    \end{align*}
    }
\caption{The Zelevinsky permutation matrix associated to the lacing diagram $\mathbf{w}$ from Figure \ref{fig:laces}.}
\label{fig:Zperm}
\end{figure}

\subsection{Pipe dreams} \label{sec:pipes}

A \textit{pipe dream} on a $k\times \ell$ grid is a subset $P \subseteq \{1, \dots, k\} \times \{1, \dots, \ell\}$. It is diagrammatically represented by tiling a $k \times \ell$ grid of squares with two types of tiles: place a \emph{cross tile} $\textcross$ in square $(i,j)$ if $(i,j)\in P$ and place an \emph{elbow tile} $\textturn$ in square $(i,j)$ if $(i,j)\notin P$. Let $|P|$ denote the number of elements (equivalently, cross tiles) in $P$. Let $\tau_1, \tau_2, \dots$ be the Coxeter generators of $S_{\infty}$ which agree with simple transpositions defined previously. 
A pipe dream $P$ yields a word 
\[\tau_{i_1+j_1-1}\tau_{i_2+j_2-1}\cdots \tau_{i_r+j_r-1}\,,\]
where $|P| = r$ and $(i_k,j_k)$ is the location of the $k^{\rm{th}}$ cross tile that appears in $P$ when reading along rows, from right to left, starting from the the top row and proceeding downwards. 

Evaluation of this word subject to the relations 
\begin{align*}
\tau_i^2=\tau_i, \quad \tau_i \tau_{i+1}\tau_i=\tau_{i+1}\tau_i\tau_{i+1}, \quad \tau_i\tau_j=\tau_j\tau_i \quad \text{for} \, |i-j| \geq 2 \,,
\end{align*}
yields a permutation $\delta(P) \in S_{k +\ell}$ called the \textit{Demazure product} (see, e.g., \cite[\S2]{Miller}). Given a permutation $v$, we say that $P$ is a \textit{pipe dream for} $v$ if $\delta(P)=v$. We say that $P$ is a \textit{reduced pipe dream} for $v$ if $P \in \Pipes(v)$ and $|P|=\ell(v)$. Equivalently, $P$ is reduced if and only if any pair of pipes crosses at most once \cite[Definition 16.2]{MillerSturmfels2005}. We denote the set of pipe dreams for a permutation $v \in S_m$ defined on a $m\times m$ grid by $\Pipes(v)$, and the subset of reduced pipe dreams by $\RPipes(v)$.

 For a permutation $v \in S_d$, denote by $\Pipes(v_0, v)$ the subset of $\Pipes(v)$ consisting of pipe dreams $P$ such that all cross tiles of  $P$ are contained within the northwest justified $d_y \times d_x$ subgrid of the $d \times d$ grid\footnote{This notation matches that found in \cite{KKR}. There, $v_0$ was the permutation in $S_d$ written in one line notation as $(d_x+1)(d_x+2)\cdots d~1~2\dots d_x$, and $\Pipes(v_0,v)$ was defined as the set of pipe dreams for $v$ supported on the Rothe diagram of $v_0$.}. There is clearly a bijection between $\Pipes(v_0,v)$ and the set of pipe dreams for $v$ on a $d_y\times d_x$ grid, so when drawing elements of $\Pipes(v_0,v)$, we only draw the northwest $d_y\times d_x$ part (see Figure \ref{fig:Pipes}). Finally, let $P_*$ be the pipe dream on a $d\times d$ grid such that $(i,j) \in P_*$ if and only if $(i,j)$ is simultaneously inside the northwest justified $d_y\times d_x$ subgrid and is \emph{outside} of the snake region (see Figure \ref{fig:Pipes}). If $P$ is a reduced pipe dream for $v\in S_m$ supported on a $m\times m$ grid, then we may recover the matrix of $v$ by putting a 1 in matrix entry $(i,j)$ if and only if the pipe entering the west wall of $P$ in row $i$ exits the north wall of $P$ in column $j$ \cite[Lemma 2.21]{KKR}. 

If $w$ is a $k \times \ell$ partial permutation matrix, then we define $\Pipes(w)$ to be all $k \times \ell$ pipe dreams $P$ such that $\delta(P)=c(w)$.\footnote{Here we are potentially extending $c(w)$ to an element of $S_{k+\ell}$ as indicated in Remark \ref{rem:completions}} We denote by $\RPipes(w) \subset \Pipes(w)$ the subset of reduced pipe dreams for $w$. 

Given $P \in \Pipes(w)$ (which by definition lies on a $k \times \ell$ grid), we can obtain a pipe dream for $c(w)$ by extending $P$ to an $m \times m$ grid by adding in elbow tiles in all positions $(i,j)$ where $i>k$ or $j >\ell$. In fact, as the following lemma shows, this operation of ``extension via elbow tiles'' produces all pipe dreams for $c(w)$. 

\begin{lemma}\label{lem:completedPipes}
    Let $w$ be a $k \times \ell$ partial permutation matrix and let $c(w) \in S_m$ be its minimal length completion. Then, the extension via elbow tiles map gives us identifications
    \begin{align*}
        \RPipes(w)=\RPipes(c(w)) \quad\quad \Pipes(w)=\Pipes(c(w))\,.
    \end{align*}
    Moreover, if $v \in S_{m'}$ is any extension of $c(w)$ by the identity permutation as in Remark \ref{rem:completions}, then $\RPipes(w)=\RPipes(v)$ and $\Pipes(w)=\Pipes(v)$ as well. 
\end{lemma}
\begin{proof}

It is a standard fact \cite[Exercise 16.2]{MillerSturmfels2005} that if $P \in \RPipes(c(w))$ then all of the cross tiles in $P$ occur in the northwest $k \times \ell$ subgrid. In other words, extension via elbow tiles is a surjective map $\RPipes(w)\to \RPipes(c(w))$ which is clearly injective as well, proving the first equality. 

To see that $\Pipes(w)=\Pipes(c(w))$ we recall from \cite[Lemma 2]{Miller} that every $P \in \Pipes(c(w))$ can be written as a union of elements of $\RPipes(c(w))$. But every element of $\RPipes(c(w))$ is an extension via elbow tiles of an element of $\RPipes(w)$, and therefore has cross tiles only in its northwest $k \times \ell$ subgrid. Therefore every cross tile of $P$ is contained in its northwest $k \times \ell$ subgrid as well. Let $P'$ be the restriction of $P$ to its northwest $k \times \ell$ subgrid. Then by definition of $\Pipes(w)$ we have $P' \in \Pipes(w)$ and $P$ is obtained from $P'$ by extension via elbow tiles. This completes the proof that $\Pipes(w)=\Pipes(c(w))$.

The final claim follows from what we have already shown. Indeed, $v \in S_{m'}$ is the minimal extension of the block matrix 
\begin{align*}
\begin{bmatrix}
    \quad c(w) & | & \text{0}_{m\times (m'-m)} \quad
\end{bmatrix}
\end{align*}
where $0_{m \times (m'-m)}$ is the $m \times (m'-m)$ matrix with only zero entries. 
\end{proof}

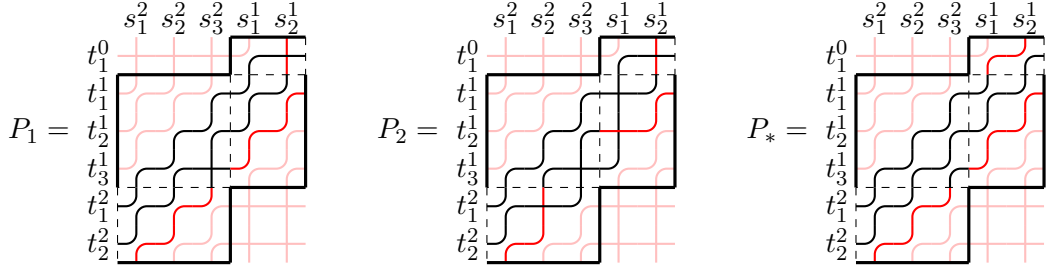
\begin{figure}
$P_1 = \vcenter{\hbox{\begin{tikzpicture}[scale=1,>=latex]
    \pipedream{0.5}{(0,0)}{$s^2_1$,$s^2_2$,$s^2_3$,$s^1_1$,$s^1_2$}
    {$t^0_1$,$t^1_1$,$t^1_2$,$t^1_3$,$t^2_1$,$t^2_2$}
    {%Always + part
      1/1/1/pink/pink,1/2/1/pink/pink,1/3/1/pink/pink,
      5/4/1/pink/pink,5/5/1/pink/pink,
      6/4/1/pink/pink,6/5/1/pink/pink,
      %variable pipes
      1/4/0/pink/black,1/5/1/red/black,
      2/1/0/pink/pink,2/2/0/pink/pink,2/3/0/pink/black,2/4/0/black/black,2/5/0/black/red,
      3/1/0/pink/pink,3/2/0/pink/black,3/3/0/black/black,3/4/0/black/red,3/5/0/red/pink,
      4/1/0/pink/black,4/2/0/black/black,4/3/1/black/black,4/4/0/red/pink,4/5/0/pink/pink,
      5/1/0/black/black,5/2/0/black/red,5/3/0/red/pink,
      6/1/0/black/red, 6/2/0/red/pink, 6/3/0/pink/pink}
    %outline of snake
    \latticepath{0.5}{(2.5,-0.5)}{very thick}{0/-3/black,-2/0/black,0/-2/black,-3/0/black}
    \latticepath{0.5}{(2.5,0)}{very thick}{-2/0/black, 0/-1/black, -3/0/black, 0/-3/black}
    \latticepath{0.5}{(2.5,0)}{dashed}{0/-1/black, -2/0/black, 0/-3/black, -3/0/black, 0/-2/black}
  \end{tikzpicture}}}$
  \qquad 
 $P_2 = \vcenter{\hbox{\begin{tikzpicture}[scale=1,>=latex]
    \pipedream{0.5}{(0,0)}{$s^2_1$,$s^2_2$,$s^2_3$,$s^1_1$,$s^1_2$}
    {$t^0_1$,$t^1_1$,$t^1_2$,$t^1_3$,$t^2_1$,$t^2_2$}
    {%Always + part
      1/1/1/pink/pink,1/2/1/pink/pink,1/3/1/pink/pink,
      5/4/1/pink/pink,5/5/1/pink/pink,
      6/4/1/pink/pink,6/5/1/pink/pink,
      %variable pipes
      1/4/0/pink/black,1/5/1/red/black,
      2/1/0/pink/pink,2/2/0/pink/pink,2/3/0/pink/black,2/4/1/black/black,2/5/0/black/red,
      3/1/0/pink/pink,3/2/0/pink/black,3/3/0/black/black,3/4/1/black/red,3/5/0/red/pink,
      4/1/0/pink/black,4/2/0/black/black,4/3/1/black/black,4/4/0/black/pink,4/5/0/pink/pink,
      5/1/0/black/black,5/2/1/red/black,5/3/0/black/pink,
      6/1/0/black/red, 6/2/0/red/pink, 6/3/0/pink/pink}
    %outline of snake
    \latticepath{0.5}{(2.5,-0.5)}{very thick}{0/-3/black,-2/0/black,0/-2/black,-3/0/black}
    \latticepath{0.5}{(2.5,0)}{very thick}{-2/0/black, 0/-1/black, -3/0/black, 0/-3/black}
    \latticepath{0.5}{(2.5,0)}{dashed}{0/-1/black, -2/0/black, 0/-3/black, -3/0/black, 0/-2/black}
  \end{tikzpicture}}}$
  \qquad 
 $P_* = \vcenter{\hbox{\begin{tikzpicture}[scale=1,>=latex]
    \pipedream{0.5}{(0,0)}{$s^2_1$,$s^2_2$,$s^2_3$,$s^1_1$,$s^1_2$}
    {$t^0_1$,$t^1_1$,$t^1_2$,$t^1_3$,$t^2_1$,$t^2_2$}
    {%Always + part
      1/1/1/pink/pink,1/2/1/pink/pink,1/3/1/pink/pink,
      5/4/1/pink/pink,5/5/1/pink/pink,
      6/4/1/pink/pink,6/5/1/pink/pink,
      %variable pipes
      1/4/0/pink/red,1/5/0/red/black,
      2/1/0/pink/pink,2/2/0/pink/pink,2/3/0/pink/black,2/4/0/black/black,2/5/0/black/red,
      3/1/0/pink/pink,3/2/0/pink/black,3/3/0/black/black,3/4/0/black/red,3/5/0/red/pink,
      4/1/0/pink/black,4/2/0/black/black,4/3/0/black/black,4/4/0/red/pink,4/5/0/pink/pink,
      5/1/0/black/black,5/2/0/black/red,5/3/0/red/pink,
      6/1/0/black/red, 6/2/0/red/pink, 6/3/0/pink/pink}
    %outline of snake
    \latticepath{0.5}{(2.5,-0.5)}{very thick}{0/-3/black,-2/0/black,0/-2/black,-3/0/black}
    \latticepath{0.5}{(2.5,0)}{very thick}{-2/0/black, 0/-1/black, -3/0/black, 0/-3/black}
    \latticepath{0.5}{(2.5,0)}{dashed}{0/-1/black, -2/0/black, 0/-3/black, -3/0/black, 0/-2/black}
  \end{tikzpicture}}}$
\caption{$P_1,P_2\in \Pipes(v_0,v(\Omega))$, where $v(\Omega)$ is the permutation from Figure \ref{fig:Zperm}. $P_1$ is reduced while $P_2$ is not.  $P_*$ is defined above. The snake region is outlined with bold black lines. }
\label{fig:Pipes}
\end{figure}

\subsection{Permutations and laces}\label{sec:lacesandperms}
Define symmetric groups associated to the arrows $\alpha_k, \beta_k$ of the bipartite quiver $Q$ by $S_{\alpha_k} := S_{\bd(y_{k-1})+\bd(x_k)}$ and $S_{\beta_k} := S_{\bd(y_k)+\bd(x_k)}$. Let $S_{\bd} = \prod_{k=1}^n (S_{\beta_k}\times S_{\alpha_k})$ 
and label a typical element by $\bv = (v_n, v^n,\dots, v_1,v^1)$, with $v_k\in S_{\beta_k}$ and $v^k\in S_{\alpha_k}$. We define $\ell(\bv)=\sum_i \ell(v_i) + \ell(v^i)$. Given a lacing diagram $\mathbf{w}=(w_n, w^n, \dots, w_1, w^1)$, we may view it as an element of $S_\mathbf{d}$ by completing it to $(c(w_n),c(\text{rot}(w^n)), \dots, c(w_1),c(\text{rot}(w^1)) \in S_\mathbf{d}$. We encode this operation as a function $\fc : \Laces(\bd) \to S_{\mathbf{d}}$ sending 
\begin{align}\label{eqn:permOperator}
    (w_n, w^n, \dots, w_1, w^1) &\mapsto (c(w_n), c(\rot(w^n)), \dots, c(w_1), c(\rot(w^1))),
\end{align}
Here, $\Laces(\bd)$ is the set of all possible lacing diagrams for the type $A$ quiver $Q$ with dimension vector $\bd$.
Note that we are viewing $c(w_k)$ as an element of $S_{\alpha_k}$ and $c(\rot(w^k)) \in S_{\beta_k}$ as in Remark \ref{rem:completions}.

If $\bv \in S_\bd$ is of the form $\fc(\bw)$ for some $\bw \in \Laces(\bd)$, we can recover the original sequence of partial permutations defining $\bw$ by a suitable ``truncation''. In detail, we define a map $LD: S_{\bd} \to \Laces(\bd)$ by 
\begin{align}
    LD(\bv):=(t_n(v_n), \rot(t^n(v^n)), \dots, t_1(v_1), \rot(t^1(v^1))),
\end{align}
where if $v \in S_{\bd(x_k)+\bd(y_k)}$ then $t_k(v)$ is the partial permutation matrix obtained by restricting the permutation matrix of $v$ to the northwest $\bd(y_k) \times \bd(x_k)$ submatrix. Similarly, if $v \in S_{\bd(y_{k-1}) + \bd(x_{k})}$, then we define $t^k(v)$ to be the partial permutation obtained by restricting to the northwest $\bd(y_{k-1})\times \bd(x_k)$ submatrix.  The following lemma follows immediately from the definitions.

\begin{lemma}\label{lem:LDPerm}
 The map $\fc: \Laces(\bd) \to S_{\bd}$ is injective, with left inverse $LD$. That is, $LD(\fc(\bw))=\bw$ for all $\bw \in \Laces(\bd)$. 
\end{lemma}

\subsection{Relating pipes and laces}\label{sec:pipestolaces} 

Let $\Pipes(\bd)$ denote the set of all pipe dreams on a $d_y\times d_x$ grid which contain $P_*$. If $P \in \Pipes(\bd)$ and $v=\delta(P)$ then we will also view $P$ as an element of $\Pipes(v_0, v)$, in particular, as a pipe dream on a $d\times d$ grid. Let $P^k$ be the \emph{mini pipe
dream} on the $\bd(y_{k-1}) \times \bd(x_k)$ grid extracted from $P\in \Pipes(\bd)$ by restricting $P$ to the block of the snake
region indexed by $\alpha_k$, and let $P_k$ be the mini pipe dream on the $\bd(y_k)\times \bd(x_k)$ grid extracted from $P$ by restricting to the block of the snake region indexed by $\beta_k$. Given a pipe dream $P \in \Pipes(\bd)$, we define the operation $\pi(P) =\mathbf{v} \in S_\mathbf{d}$ by $\mathbf{v}=(v_n, v^n, \dots, v_1, v^1)$ where $v_k=\delta(P_k)$ and $v^k=\delta( \rot(P^k))$. 

We can relate pipe dreams to lacing diagrams in the following way. For integers $k,\ell >0$ and any permutation $v$, we define the ``\textit{follow the pipes}'' operation $\fw$ mapping $k \times \ell$ pipe dreams in $\Pipes(v)$\footnote{By this we mean pipe dreams $P \in \Pipes(v)$ whose cross tiles appear only in the northwest $k \times \ell$ subgrid.} to $k \times \ell$ partial permutation matrices, as follows: given $P' \in \RPipes(v)$ on a $k \times \ell$ grid, we obtain a partial permutation matrix by following the pipes from the left boundary to the top boundary. That is, $\fw(P)$ has a $1$ in entry $(i,j)$ precisely when there is a pipe running from the $i^{\rm{th}}$ row of $P$ to the $j^{\rm{th}}$ column. If $P \in \Pipes(v)$ is non-reduced, then whenever a pair of pipes crosses more than once we remove the cross tiles at their intersection points (starting northeast and moving southwest) until they have a unique crossing, and then read the resulting partial permutation as before. Observe that if $P \in \RPipes(w)$ for some partial permutation $w$ then $\fw(P)=w$. We extend this to an operation $\mathfrak{w}$, called ``\emph{pipes to laces}'', on a $d_y \times d_x$ pipe dream $P$ with $P_* \subseteq P$ 
by defining $\mathfrak{w}(P)$ to be the lacing diagram 
\begin{align*}
\fw(P):=(\mathfrak{w}(P_n),  \rot(\mathfrak{w}(\rot(P^n))), \dots,\mathfrak{w}(P_1),\rot(\mathfrak{w}(\rot(P^1))))\,,
\end{align*}
implicitly identifying partial permutations with their lacing diagrams. Note that this definition differs from the definition of the pipes to laces operation in  \cite[\S~2.9]{KKR}, as in \textit{loc. cit.} the authors only consider the case where each $P_k$ and $P^k$ is reduced. Our definition is the extension of theirs to the case of non-reduced pipe dreams, where a double crossing of pipes in a $P_k$ or $\rot(P^k)$ is simply ignored. 

\begin{remark}
    Our color conventions for pipes in pipe dreams follow those of \cite[\S4]{KKR}. These conventions ensure that pipes go to laces in the completed lacing diagram of the same color under the pipes to laces operation. 
\end{remark}

\begin{example}
The pipe dream $P_1$ in Figure \ref{fig:Pipes} corresponds, under pipes to laces, to the lacing diagram in Figure \ref{fig:laces}. It is reduced as it does not contain double crossings. The pipe dream $P_2$ is non-reduced (though each mini pipe dream extracted from $P_2$ is reduced), and corresponds to the second lacing diagram  in Figure \ref{fig:KTheoryLaces} under the pipes to laces operation.
\end{example}

We record here a few combinatorial results from \cite{KKR} which will be used in later sections.

\begin{proposition}\label{prop:followpipes}
If $P\in \Pipes(v_0,v(\Omega))$, then $\fc(\mathfrak{w}(P))= \pi(P)$ as elements of $S_{\bd}$.
\end{proposition}

\begin{proof}
    The proof of this proposition follows \emph{mutatis mutandis} from the proof of \cite[Proposition 4.15]{KKR}.
\end{proof}

\begin{corollary}\label{lem:partialperm}
If $P\in \Pipes(v_0, v(\Omega))$ and $\pi(P)=(v_n, v^n, \dots, v_1, v^1)$, then $v^k = \delta(\rot(P^k))$ and $v_k = \delta(P_k)$ are completions of partial permutation matrices.
\end{corollary}

\begin{proof}
    This is immediate, since by Proposition \ref{prop:followpipes}, $v_k$ is the completion of the partial permutation $\fw(P_k)$ and $v^k$ is the completion of the partial permutation $\fw(\rot (P^k))$.
\end{proof}

\begin{remark}
    For reduced pipe dreams $P \in \RPipes(v_0, v(\Omega))$ we will prove in Proposition \ref{prop:wPminLacingDiagram} that $\fw(P) \in W(\Omega)$. In particular, $\fw(P)$ is a lacing diagram in the open orbit $\Omega^\circ$ (see also \cite[Theorem 4.16]{KKR}). For non-reduced pipe dreams, we will prove in Proposition \ref{prop:lacesBijection} that $\fw(P) \in KW(\Omega)$. Note however, that $K$-theoretic lacing diagrams $\bw \in KW(\Omega)$ \emph{need not lie in the open orbit} $\Omega^\circ$. That is, the number of laces connecting any two columns in $\bw \in KW(\Omega)$ may differ from the number of laces between those columns in a minimal lacing diagram $\bw' \in W(\Omega)$. For an example of this phenomenon, see \cite[Figure 4]{KinserICRA}.
\end{remark}

\subsection{Schubert and Grothendieck polynomials in terms of pipe dreams}\label{sect:Alpha}

Our formulas for $K$-polynomials and multidegrees of type-$A$ quiver loci involve double \textit{Grothendieck polynomials} and \textit{Schubert polynomials}. These polynomials, denoted $\mathfrak{G}_v(\mathbf{a}; \mathbf{b})$ and $\mathfrak{S}_v(\mathbf{a};\mathbf{b})$ respectively, are defined with respect to a permutation $v \in S_m$ and two alphabets $\mathbf{a}=a_1, \dots, a_m$ and $\mathbf{b}=b_1, \dots, b_m$. The original definitions of these polynomials were recursive, and can be found in \cite{FultonLascoux}. The most useful characterization of these polynomials for our purposes is in terms of pipe dreams: let $P \in \Pipes(v)$ be a pipe dream for a permutation $v \in S_{m}$. Letting the alphabets $\mathbf{a}=a_1, \dots, a_m$ and $ \mathbf{b}=b_1, \dots , b_m$ index the rows and columns of $P$ respectively, we define 
\begin{equation}
(\mathbf{a}-\mathbf{b})^P=\prod_{(i,j) \in P} (a_i-b_j), 
\qquad \left(1-\frac{\mathbf{a}}{\mathbf{b}}\right)^P= \prod_{(i,j) \in P} \left( 1 - \frac{a_i}{b_j} \right).
\end{equation}

The following formulas are a special case of those presented in \cite[Theorem 4.5]{WooYong}, and previously appeared in different forms in \cite{AJS1994,Billey1999,GK2015,Willems2006}.

\begin{proposition}\label{prop:pipedreamGrothendieck}
For a permutation $v \in S_m$ and alphabets $\mathbf{a}=a_1, \dots, a_m, \mathbf{b}=b_1, \dots, b_m$, the following formulas hold:
\begin{align*}
\mathfrak{G}_v(\mathbf{a}; \mathbf{b})&=\sum_{P \in \Pipes(v)} (-1)^{|P|-\ell(v)} \left( 1- \frac{\mathbf{a}}{\mathbf{b}} \right)^{P} \\
\mathfrak{S}_v(\mathbf{a};\mathbf{b})&=\sum_{P \in \RPipes(v)} \left(\mathbf{a} -\mathbf{b}\right)^P.
\end{align*}
\end{proposition}

If $w$ is a $k \times \ell$ partial permutation matrix, then we define $\fG_w(\ba;\bb):=\fG_{c(w)}(\ba; \bb)$ and $\fS_w(\ba;\bb):=\fS_{c(w)}(\ba;\bb)$ for alphabets $\ba= a_1, \dots, a_{k + \ell}$ and $\bb=b_1, \dots, b_{k + \ell}$. By Lemma \ref{lem:completedPipes} we then have 
\begin{align*}
    \mathfrak{G}_w(\mathbf{a}; \mathbf{b})&=\sum_{P \in \Pipes(w)} (-1)^{|P|-\ell(c(w))} \left( 1- \frac{\mathbf{a}}{\mathbf{b}} \right)^{P} \\
\mathfrak{S}_w(\mathbf{a};\mathbf{b})&=\sum_{P \in \RPipes(w)} \left(\mathbf{a} -\mathbf{b}\right)^{P}\,,
\end{align*}
where the sums are taken over $k \times \ell$ pipe dreams for $w$. Note that by Lemma \ref{lem:completedPipes}, only the variables $a_1, \dots, a_k$ and $b_1, \dots, b_\ell$ will appear in these expressions.

The following notation will be used in our proofs of the cohomological and $K$-theoretic component formulas (Theorem \ref{thm:3.9} and Theorem \ref{thm:4.16}). To a lacing diagram $\bw=(w_n, w^n, \dots, w_1, w^1)$ we assign the following polynomials: 
\begin{align*}
    \fS_{\bw}(\bt; \bs) &= \left( \prod_{k=1}^n \fS_{w_k}(\bt^k;\bs^k) \right) \cdot \left( \prod_{k=1}^n \fS_{\rot(w^k)}(\tilde{\bt}^{k-1}; \tilde{\bs}^k) \right) \\
    \fG_{\bw}(\bt; \bs) &= \left( \prod_{k=1}^n \fG_{w_k}(\bt^k;\bs^k) \right) \cdot \left( \prod_{k=1}^n \fG_{\rot(w^k)}(\tilde{\bt}^{k-1}; \tilde{\bs}^k) \right) \, , \\
\end{align*}
 where if $\ba=a_1, \dots, a_m$ is an alphabet, we denote the reverse alphabet $a_m, \dots, a_1$ by $\tilde{\ba}$. Recall that the alphabets $\bt^k$ and $\bs^k$ are defined in the introduction.

\section{The bipartite codimension formula and the bipartite pipe formulas}\label{sect:codim}

In this section we give a combinatorial proof of Buch and Rim\'anyi's formula for the codimension of a bipartite type $A$ quiver locus $\Omega$, in its representation space $\rep_Q(\bd)$, in terms of number of crossings in the extended diagram of any $\bw \in W(\Omega)$ (see \cite[Corollary 2]{BuchRim}): \[\codim (\Omega)=|\bw|, \quad \bw \in W(\Omega).\] 
We note that Buch and Rim\'anyi's formula holds for arbitrarily oriented type $A$ quivers, while our proof is specific to the bipartite setting. We also remark that as a consequence of the results needed to prove the codimension formula, we obtain a combinatorial proof of the bipartite pipe formula first proven in \cite[Theorem 3.7]{KKR} (see Theorem \ref{thm:pipeformulas} below).

Buch and Rim\'anyi's codimension formula was a corollary of their cohomological component formula, which was proven via geometric methods. Because one of the main goals of this paper is to give a new combinatorial proof of the cohomological component formula, we require an independent combinatorial proof of this codimension formula. Our first step in this direction is to prove Proposition \ref{prop:wPminLacingDiagram} below. While this result was also proven in \cite{KKR}, their proof made use of \cite[Corollary 2]{BuchRim}. Hence we provide a new, independent proof.

We need the following lemma which was proven in \cite{KKR} by a Gr\"obner degeneration argument. Here we give an alternate combinatorial proof. Recall that $P_*$ is the pipe dream on a $d\times d$ grid defined by $(i,j) \in P_*$ if and only if $(i,j)$ is simultaneously inside the northwest justified $d_y\times d_x$ subgrid and is \emph{outside} of the snake region (see Figure \ref{fig:Pipes}). Recall that $v_*$ denotes the Zelevinsky permutation associated to the entire representation space $\text{$\rep$}_Q(\mathbf{d})$.

\begin{lemma}\cite[Lemma 3.5]{KKR}\label{lemma:3.2}
Each $P \in \Pipes(v_0, v(\Omega))$ contains $P_*$ as a subset. 
\end{lemma}

\begin{proof}
    Let $P\in \Pipes(v_0,v(\Omega))$ and suppose that $P$ has at least one elbow tile strictly northwest of the snake region. Then there is some maximally northwest such elbow tile in this region, say in location $(i,j)$. Thus, there is a pipe in $P$ which enters at the left of the $d_y\times d_x$ grid in row $i$ and exits at the top of the $d_y\times d_x$ grid in column $j$. This means that $v(\Omega)$ has a $1$ in location $(i,j)$. But this is impossible by properties (\textit{Z1}) and (\textit{Z2}) of the Zelevinsky permutation: there are no $1$s in $v(\Omega)$ that are simultaneously in the northwest $d_y\times d_x$ submatrix of $v(\Omega)$ and strictly northwest of the snake region.    

Similarly, suppose that $P$ has at least one elbow tile that is simultaneously strictly southeast of the snake region and contained in the northwest $d_y\times d_x$ subgrid. Then, there is a maximally southeast such elbow tile in this region, say in row $a$ and column $b$. Thus, there is a pipe in $P$ which enters at the bottom of the northwest $d_y\times d_x$ subgrid in column $b$ and exits at the right of the northwest $d_y\times d_x$ subgrid in row $a$. Column $b$ appears in a block column labeled by $x_k$ and row $a$ appears in a block row labeled by $y_\ell$, where $\ell>k$. Thus, $P$ has a pipe which starts at the left of the $d\times d$ grid in block row $x_k$ and ends at the top of the $d\times d$ grid in block column labeled by $y_\ell$, $\ell>k$. Thus, $v(\Omega)$ has a $1$ in the intersection of block row $x_k$ and block column $y_\ell$. But, this violates properties (\textit{Z1}) and (\textit{Z2}) of the Zelevinsky permutation since vertex $x_k$ neither appears to the left of vertex $y_\ell$, nor does vertex $x_k$ appear exactly one vertex to the right of vertex $y_\ell$ (since $\ell>k$). This completes the proof that each $P\in \Pipes(v_0,v(\Omega))$ contains $P_*$ as a subset.
\end{proof}

The following proposition is the key combinatorial result needed to prove the codimension formula.

\begin{proposition}\label{prop:wPminLacingDiagram}\cite[Proposition 4.16]{KKR} \label{prop:PtoL}
If $P \in \RPipes(v_0, v(\Omega))$, then $\mathfrak{w}(P) \in W(\Omega)$. 
\end{proposition}

\begin{proof}
By the first part of the proof of \cite[Proposition 4.16]{KKR}, $\mathfrak{w}(P)$ is a lacing diagram in the orbit $\Omega^\circ$. Note that this part of the proof is combinatorial, and does not make use of the codimension formula. It remains to show that $\mathfrak{w}(P)$ is a \emph{minimal} lacing diagram. Let $\bw = \mathfrak{w}(P)$.

By the remarks after the proof of \cite[Lemma 1]{BuchRim}, showing that $\bw$ is minimal is equivalent to showing that $\bw$ satisfies the following two conditions:
    \begin{enumerate}
        \item no strands in the extended lacing diagram cross twice; and 
        \item if two strands in $\bw$ start or end in the same column, then the corresponding strands in $\bw$'s extended diagram do not cross. 
    \end{enumerate}
    We will first prove point (1). Note that no two strands can cross twice in $\bw$, as this would mean that a pair of pipes in $P$ crossed twice in the snake region, and $P$ would not be reduced. We must therefore show that strands do not cross twice, even in the extended diagram. To do this, we will show that crossings in the extended diagram correspond to crossings of pipes in the snake region of $P$. 
    
    So, suppose that there is a crossing in the extended diagram from $y_i$ to $x_i$ (i.e. this is a crossing of arrows which point to the right). By Proposition \ref{prop:followpipes} we have that $\pi(P)=\fc(\bw) \in S_{\bd}$. Therefore $\delta(P_i)=c(w_i) \in S_{\bd(x_i)+\bd(y_i)}$. (Recall that $P_i$ is the mini pipe dream on the $\bd(y_i)\times \bd(x_i)$ grid extracted from $P$ by restricting to the block of the snake region indexed by $\beta_i$.) Let $\tilde{P}_i$ denote the pipe dream which is the extension-by-elbow-tiles of $P_i$ to a $\bd(x_i) + \bd(y_i)$ square grid. Let $\tilde{w}_i=\fw(\tilde{P}_i)$ be the permutation obtained by pipes to laces. Then we have that
\begin{align*}
    \tilde{w}_i= \delta(\tilde{P}_i) = \delta(P_i) = c(w_i),
\end{align*}
that is, we can determine $c(w_i)$ by following the pipes in the pipe dream $\tilde{P}_i$. It follows that a crossing of right pointing arrows in $c(w_i)$ which  appears in the extended lacing diagram appears in $\tilde{P}_i$ as a crossing of pipes. However, by Lemma \ref{lem:completedPipes}, this means that these pipes crossed in $P_i$ as well. A similar argument shows that a crossing of left pointing arrows in the completed lacing diagram will appear in $P$ as a crossing of pipes. We have therefore proven that a crossing of laces in the extended lacing diagram corresponds to a crossing of the corresponding pipes in $P$. Since $P$ is reduced, it follows that no two laces can cross twice, even in the extended diagram. 

To prove point (2), let $\ell_1$ and $\ell_2$ be laces which both have their left endpoints indexed by vertex $x_i \in Q_0$. The laces $\ell_1$ and $\ell_2$ correspond to pipes in $P$ which enter the snake region through the south of block $\beta_i$. Therefore, since $P_\ast\subseteq P$ (see Lemma \ref{lemma:3.2}), these pipes enter $P$ at the left boundary of the $d\times d$ grid in block row $x_i$. 
Now assume that $\ell_1$ and $\ell_2$ cross in $\bw$, so that the corresponding pipes cross in $P$. 
Then this introduces a configuration of $1$s in $\delta(P) = v(\Omega)$ which do not run northwest to southeast in block row $x_i$. As no such configuration of $1$s exists in $v(\Omega)$, we conclude that the laces $\ell_1$ and $\ell_2$ did not cross. Similarly, laces which start at a vertex $y_i$ cannot cross.

The only remaining case to consider is when $\ell_1$ and $\ell_2$ do not cross in $\bw$, but potentially cross in the extended lacing diagram. These laces may not cross to the left of vertex $x_i$ in the extended diagram, as this would violate the fact that $\bw$'s completed diagram is the minimal length extension of $\bw$. Similarly, if $\ell_1$ and $\ell_2$ end at the same vertex, they may not cross to the right of this vertex, and this would contradict minimality as well. We are left considering the case in which $\ell_1$ ends at vertex $z_1$ and $\ell_2$ ends at vertex $z_2$, say with $z_2$ to the right of $z_1$ in $Q$. If the corresponding laces in the extended diagram crossed, then at this crossing, $\ell_2$ is non-virtual. It follows that this crossing would appear in $P$, and we conclude as before that $\delta(P)$ would violate the conditions of the Zelevinsky permutation. 

Finally, let $m_1$ and $m_2$ be laces which both have their right endpoints indexed by vertex $z \in Q_0$. By an analogous argument as given above, if $m_1$ and $m_2$ were to cross, this would introduce a configuration of $1$s in $v(\Omega)$ which do not run northwest to southeast in the block column labeled by $z$. Again, no such configuration of $1$s exists in $v(\Omega)$, and so these laces did not cross. The cases in which the laces $m_1$ and $m_2$ do not cross in $\bw$ but potentially cross in the extended diagram are ruled out by a similar analysis to the case considered above, where $\ell_1$ and $\ell_2$ stated at the same vertex in $Q$.
This completes the proof of point 2, and thus the proof of the proposition.
\end{proof}

\begin{lemma}\cite[Lemma 3.5]{KKR}\label{lemma:pipeformula}
Let $v_\ast$ denote the Zelevinsky permutation of the bipartite type $A$ representation space $\rep_Q(\bd)$. Then, $\Pipes(v_0, v_*)=\{P_*\}$. 
\end{lemma}

\begin{proof}
By Lemma \ref{lemma:3.2} each element of $\RPipes(v_0,v_*)$ has $P_*$ as a subset. Furthermore, if $P\in \RPipes(v_0,v_*)$, then by Proposition \ref{prop:PtoL} we have that $\fw(P)\in W(\rep_Q(\bd))$. 
But $W(\rep_Q(\bd)) = \{\mathfrak{w}(P_*)\}$. 
The only way for $P$ to simultaneously contain $P_*$ and satisfy $\mathfrak{w}(P) = \mathfrak{w}(P_*)$ is if either $P = P_*$ or if $P$ contains a double crossing in a single block of the snake region (if there are only ever single crossings, then $\mathfrak{w}(P)$ would have a crossing). But then $P$ would not be a reduced pipe dream. Thus, $P = P_*$ and we conclude that $\RPipes(v_0,v_*) = \{P_*\}$. By \cite[Lemma 2]{Miller}, $\Pipes(v_0,v_*) = \{P_*\}$ as well.
\end{proof}

We are now ready to complete our combinatorial proof of Buch and Rim\'anyi's formula for the codimension of a bipartite type $A$ quiver locus in its representation space. As noted earlier, Buch and Rim\'anyi's formula holds more generally (i.e., for arbitrarily oriented type $A$ quivers).

\begin{corollary}\label{cor:codim}\cite[Corollary 2]{BuchRim}
    Let $\Omega$ be a quiver locus. Then for any minimal lacing diagram $\bw \in W(\Omega)$, we have $|\bw|=\codim\Omega$.
\end{corollary}

\begin{proof} 
    By \cite[Equation 2.12]{KKR} (which follows easily from basic properties of the bipartite Zelevinsky map \cite{KR}), we have $\codim(\Omega)=\ell(v(\Omega))-\ell(v_\ast)$, where $\ell(v)$ denotes the length of the permutation $v$. Furthermore, $\ell(v(\Omega))-\ell(v_\ast) = |P\setminus P_\ast|$, for any $P\in \text{RPipes}(v_0,v(\Omega))$.
    
    Let $\bw=\fw(P)$ be the image of $P$ under the pipes to laces map. By Proposition \ref{prop:PtoL}, $\bw \in W(\Omega)$. Clearly $|\bw| \geq |P\setminus P_\ast|$. However, by Lemma \ref{lem:completedPipes}, all crossings in $\bw$'s extended diagram correspond to $\textcross$ tiles in $P$, so that $|\bw|=|P\setminus P_\ast|$. Therefore $|\bw|=|P\setminus P_\ast|=\codim \Omega$, as desired. The statement of corollary now follows by recalling that all elements of $W(\Omega)$ have the same number of crossings (in their extended diagrams). 
\end{proof}

\begin{figure}
\begin{align}
\begin{blockarray}{cccc|cc|c|ccc|cc}
            & s_1^2 & s_2^2 & s_3^2 & s_1^1 & s_2^1 & t_1^0 & t_1^1 & t_2^1 & t_3^1 & t_1^2 & t_2^2  \\
      \begin{block}{c[ccccc|cccccc]}
t_1^0 &  &  &  & 1 & 0 &  & \\
\cline{1-1}
t_1^1 & 1 & 0 & 0 &  &  & &   \\
t_2^1 & 0 & 1 & 0 &  &  & & &  \\
t_3^1 & 0 & 0 & 1 &  &  & & & & \\
\cline{1-1}
t_1^2 &  &  &  & & & 1 & 0 & 0 & 0 & \\
t_2^2 &  &  &  & & & 0 & 1 & 0 & 0 & &  \\
\cline{1-12}
s_1^2 &  &  &  &  0 & 1 & & & & & 0 & 0\\
s_2^2 &  &  &  &  0 & 0 &  &  &  &  &  1  &  0  & \\
s_3^2 &  &  &  &  0 & 0 &  &  &  &  &  0  &  1  &\\
\cline{1-1}
s_1^1 & & & & & & & 0 & 1 & 0 & \\
s_2^1 & & & & & & & 0 & 0 & 1 &\\
      \end{block}
    \end{blockarray}
\end{align}
\caption{The Zelevinsky Permutation associated to the lacing diagram $\mathbf{w}$ in Figure \ref{fig:laces} with the block structure now labeled by our alphabets $\mathbf{s}, \mathbf{t}$.} 
\label{fig:Zpermlabeled}
\end{figure}

Lemma \ref{lemma:pipeformula} together with Proposition \ref{prop:pipedreamGrothendieck} implies that $\mathfrak{G}_{v_\ast}(\mathbf{t}, \mathbf{s}; \mathbf{s}, \mathbf{t})$ divides $\mathfrak{G}_{v(\Omega)}(\mathbf{t}, \mathbf{s}; \mathbf{s}, \mathbf{t})$. 
Note that, in the notation of Proposition \ref{prop:pipedreamGrothendieck}, the alphabet $\ba$ is the concatenated alphabet $\ba=(\bt,\bs)$ and the alphabet $\bb$ is the concatenated alphabet $\bb=(\bs,\bt)$. Combining this with the ratio formulas (see equation \eqref{eq:ratioformulas}) yields the \textit{bipartite pipe formulas} \cite[Theorems 3.7 and 5.11]{KKR}. See \cite{KKR} for details.

\begin{theorem}\label{thm:pipeformulas}\cite[Theorem 3.7]{KKR}
For any bipartite type $A$ quiver locus $\Omega$, we have
\begin{align*}
\mathcal{Q}_{\Omega}(\bt-\bs)&=\sum_{P \in \RPipes(v_0, v(\Omega))}(\bt-\bs)^{P\setminus P_\ast}\, , \\
K\mathcal{Q}_\Omega(\mathbf{t}/\mathbf{s})&=\sum_{P \in \Pipes(v_0, v(\Omega))} (-1)^{|P\setminus P_*|-\codim \Omega} \left( 1- \frac{\mathbf{t}}{\mathbf{s}} \right)^{P \setminus P_*}\,.
\end{align*}
\end{theorem} 

\section{A combinatorial proof of the bipartite cohomological component formula}\label{sect:cohomology}
In this section we provide a combinatorial proof of the type $A$ cohomological quiver component formula. The type $A$ cohomological quiver component formula was first stated and proved by Buch and Rim\'anyi in \cite{BuchRim}. Our proof was inspired by Yong's work \cite{yongComponent}. 

As explained in Section \ref{sec:pipestolaces}, given any $P\in \Pipes(\bd)$, we can extract mini pipe dreams $P_i$ and $P^i$ from $P$ by restricting $P$ to the blocks of the snake region $\beta_i$ and $\alpha_i$ respectively and we have the operation $\pi: \Pipes(\bd)\rightarrow S_\bd$ defined in Section \ref{sec:pipestolaces}. We begin with a definition.

\begin{definition}
Let $\mathbf{w}=(w_n,w^n, \dots, w_1, w^1)$ be a lacing diagram associated to a representation in $\rep_Q(\bd)$. A pipe dream $P\in \Pipes(\bd)$ is a \textit{pipe network} for $\mathbf{w}$ if $\pi(P) = \fc(\bf{w})$. That is, $P_i \in \Pipes(w_i)$ and $\rot(P^i) \in \Pipes(w^i)$ for each $i=1, \dots, n$.  
The set of all pipe networks for $\mathbf{w}$ is denoted $\text{PipeNet}(\mathbf{w})$. A pipe network $P\in \text{PipeNet}(\mathbf{w})$ is \textit{PN-reduced} if for each $i=1, \dots, n$ one has $P_i \in \RPipes(w_i)$ and $\rot(P^i) \in \RPipes(w^i)$. Denote the set of PN-reduced pipe networks for $\bw \in W(\Omega)$ by $\RPipeNet(\bw)$.
\end{definition}

Note that a $PN$-reduced pipe network need not be reduced as a pipe dream. For example, the second pipe dream in Figure \ref{fig:Pipes} is non-reduced, but is a $PN$-reduced pipe network. However, if $\bf{w}$ is a \emph{minimal} lacing diagram, we have the following: 

\begin{lemma}\label{lem:PNtoRedPipes}
    Let $\mathbf{w}\in W(\Omega)$ be a minimal lacing diagram. If $P\in \RPipeNet(\bw)$, then $P\in \RPipes(v_0, v(\Omega))$.
\end{lemma}

\begin{proof}
    The proof will proceed in four steps. In the first step we will show that every $P \in \RPipeNet(\bw)$ is a reduced pipe dream (in the usual sense). The final three steps will demonstrate that any $P \in \RPipeNet(\bw)$ satisfies the three conditions defining the Zelevinsky permutation $v(\Omega)$. 

   First, however, we note that by the definitions of $\RPipeNet(\bw)$ and $\RPipes(w)$ (for $w$ a partial permutation), that if $P \in \RPipeNet(\bw)$, we have $\fw(P)=\bw$. This will be used extensively below. As a final remark, throughout this proof we will be viewing our $d_y \times d_x$ pipe dreams as being embedded in a $d \times d$ pipe dream in the usual way. 

    \textit{Step 1:} We first show that $P$ is a reduced pipe dream, that is, we show that no pair of pipes cross twice. By contradiction, suppose that two pipes $\ell_1$ and $\ell_2$ have a double crossing. If this double crossing happens in a single block of the snake region, say $\beta_i$, then it follows that $P_i$ is not a reduced pipe dream for $w_i$. This contradicts that $P$ is $PN$-reduced. 
So, suppose that the double intersection happens in distinct blocks of the snake region. Since $\fw(P)=\bw$, a double intersection of pipes $\ell_1$ and $\ell_2$ in distinct blocks in the snake region of $P$ results in a double intersection of two laces in the (extended) lace diagram $\mathbf{w}$. This implies that  $\mathbf{w}$ is a non-minimal lacing diagram, a contradiction.

\textit{Step 2:} Next we prove that $P$ satisfies condition (\emph{Z1}) of the Zelevinsky permutation. There are four cases we must consider, corresponding to whether $z_i$ and $z_j$ are sources or sinks. We will only demonstrate one of these cases in detail, as the arguments for the remaining cases are very similar. 

We consider the case where $z_i=y_i$ is a source and $z_j=x_j$ is a sink, where, necessarily, $i \geq j$. We will denote the number of laces from $y_i$ to $x_j$ by $n_{ij}$.\footnote{Here we are counting laces in the non-completed lacing diagram} We must show that this is equal to $n_{ij}'$, which we define to be the number of $1$s in the permutation $\delta(P)$ which lie in the intersection of block row $y_i$ and block column $x_j$. 

To see that $n_{ij} \leq n_{ij}'$, let $\ell$ be a lace in $\bw$ which starts at $y_i$ and ends at $x_j$. Let $p_\ell$ be the corresponding pipe in $P$ under the equality $\fw(P)=\bw$. Consider the segment $\kappa_\ell$ of $p_\ell$ which runs from the interface between block $\alpha_i$ and $\beta_{i-1}$ to the interface between block $\alpha_j$ and block $\beta_j$. This $\kappa_\ell$ is precisely the segment of $p_\ell$ which is identified with $\ell$ via pipes to laces. In our colouring conventions for pipe dreams, $\kappa_\ell$ is the black segment of $p_\ell$. In these colouring conventions we see that the segment of $p_\ell$ in block $\alpha_i$ is red and necessarily exits block $\alpha_i$ via its left wall. Indeed, if this were not the case then $p_\ell$ would necessarily exit via the south wall of $\alpha_i$ and $\ell$ would not terminate at $y_i$ but at some vertex to the left of $y_i$ in $Q_0$. A similar analysis applies to the segment of $p_\ell$ contained in $\alpha_j$. In our colouring conventions, this segment is coloured red and necessarily exits $\alpha_j$ via its north wall. Indeed, if this were not the case then this segment would exit via the east wall and $\ell$ would not have $x_j$ as its right endpoint, but some other vertex to the right of $x_j$ in $Q_0$. 

We have therefore shown that $p_\ell$ exits the snake region in block row $y_{i}$ and block column $x_j$. Since $P_\ast \subset P$, we see that (reading left to right) the pipe $p_\ell$ enters $P$ in block row $y_i$ and exits $P$ in block column $x_i$. Since $P$ is reduced, we can read off the permutation $\delta(P)$ by following the pipes in this way and we conclude that $n_{ij} \leq n_{ij}'$. 

The proof that $n_{ij}' \leq n_{ij}$ is very similar. Let $p$ be a pipe which enters the west wall of $P$ in block row $y_i$ and exits in block column $x_j$. By reversing the previous argument we see that $p$ produces a lace between $y_i$ and $x_j$ in $\bw$, since $\bw=\fw(P)$. This completes the proof of (\emph{Z1}). 

\textit{Step 3:} In this step we prove that $P$ satisfies condition (\emph{Z2}) of the Zelevinsky permutation. There are two cases one must consider here, depending on whether the rightmost vertex is a source or a sink. We will only prove one of these cases, as the other is very similar. So, let $z_i=x_i$ so that $z_j=y_i$\footnote{Note that in the statement of (\emph{Z2}) the vertex $z_i$ appears one space to the \textit{right} of $z_j$}. Let $N_{i}$
be the number of laces between $x_i$ and $y_i$ in $\bw$ and let $N_i'$ be the number of $1$s which appear in the corresponding block of the permutation $\delta(P)$. Since $P$ is reduced, $N_i'$ is equivalently the number of pipes which enter $P$ in block row $x_i$ and exit in block column $y_i$. 

By analyzing the block decomposition of the $d \times d$ pipe dream $P$, and using the fact that $P_\ast \subset P$, we see that $N_i'$ is equivalently the number of pipes which enter the snake region via the south wall of $\beta_i$ and exit the snake region via the east wall of $\beta_i$. Now, using the fact that $\fw(P)=\bw$ we see that $N_i$ is equal to the number of pipes which enter $\beta_i$ through its west wall and exit $\beta_i$ via its north wall. It follows that there are $\bd(y_i) - N_i$ pipes which enter $\beta_i$ through its west wall and exit $\beta_i$ through its east wall. This in turn implies that there are $\bd(x_i) - N_i$ pipes which enter $\beta_i$ through its south wall and exit $\beta_i$ via its north wall. Thus, there are $\bd(x_i)-(\bd(x_i)-N_i)=N_i$ pipes which enter $\beta_i$ in the south and exit $\beta_i$ in the east, proving $N_i=N_i'$, and completing the proof of property (\emph{Z2}).

\textit{Step 4:} Suppose that $z\in S_d$ is any permutation with the same number of $1$s in each block as $v(\Omega)$. Since the $1$s in $v(\Omega)$ appear northwest to southeast along block rows and block columns, we observe that $z = v(\Omega)$ if and only if the length of the permutation $z$ is equal to the length of $v(\Omega)$. We write $\ell(z)$ to denote the length of $z\in S_d$.

Now, by conditions ($Z1$) and ($Z2$), $\delta(P)$ and $v(\Omega)$ have the same number of $1$s in each block. So to conclude that $\delta(P) = v(\Omega)$, we will check that $\ell(\delta(P)) = \ell(v(\Omega))$.
Since $\mathbf{w}$ is a minimal lacing diagram we have that 
\begin{align*}
\codim \Omega &= \ell(\mathbf{w}) = 
\sum_{i=1}^n \ell(w_i) + \ell(w^i)
=\sum_{i=1}^n |P_i|+|\text{rot}(P^i)| 
=\sum_{i=1}^n |P_i|+|P^i| 
=|P\setminus P_*|,
\end{align*}
where the first equality is Corollary \ref{cor:codim}, and the third equality holds because each $P_i$ and $\text{rot}(P^i)$ are reduced pipe dreams. 
Now, if $Q\in \RPipes(v_0, v(\Omega))$, then $\codim \Omega = |Q\setminus P_*|$. Indeed, from \cite[Equation 2.12]{KKR} we have $\codim(\Omega)=\ell(v(\Omega))-\ell(v_\ast)$. Since $Q$ and $P_\ast$ are reduced pipe dreams we have $|Q|=\ell(v(\Omega))$ and $|P_\ast|=\ell(v_\ast)$, which implies that $\codim \Omega=|Q\setminus P_\ast|$. Thus, $|Q| = |P|$. 
Since $P$ is a reduced pipe dream, we have
\begin{equation}\label{eq:plusCompare}
\ell(\delta(P)) = |P| = |Q| = \ell(v(\Omega)),
\end{equation}
as desired.
\end{proof}

The following corollary was proven geometrically in \cite{KKR}. We will make use of it later.
\begin{corollary} \label{cor:OntoLaces}
If $\bw\in W(\Omega)$, then there exists a $P\in \RPipes(v_0,v(\Omega))$ with $\pi(P) = \fc(\bw)$. 
\end{corollary}

\begin{proof}
    Recall that to any $k\times \ell$ partial permutation matrix $w$, there is an associated reduced pipe dream $P_w$ supported on a $k\times \ell$ rectangle, satisfying $\delta(P_w) = c(w)$ \cite[\S16]{MillerSturmfels2005}. Thus, given a lacing diagram $\bw\in W(\Omega)$, there exists a reduced pipe network $P\in \RPipeNet(\bw)$. By Lemma \ref{lem:PNtoRedPipes}, $P\in \RPipes(v_0,v(\Omega))$.
\end{proof}

The following is the main result that we need to give our direct combinatorial proof of the multidegree component formula. 

\begin{proposition}\label{lem:3.7}
The following equality of sets of pipe dreams holds:
\begin{align*}
\bigcup_{\mathbf{w} \in W(\Omega)} \RPipeNet(\mathbf{w}) = \RPipes(v_0, v(\Omega)).
\end{align*}
Moreover, the union on the left side of the equality is a disjoint union.
\end{proposition}

\begin{proof}
First note that it is clear by construction that the union $\bigcup_{\mathbf{w} \in W(\Omega)} \RPipeNet(\mathbf{w})$ is a disjoint union. By Lemma \ref{lemma:3.2}, $\RPipes(v_0,v(\Omega))\subseteq \RPipes(\bd)$, so we may apply the operation $\pi$ to each $P\in \RPipes(v_0,v(\Omega))$. By Propositions \ref{prop:followpipes} and \ref{prop:PtoL}, if $P \in \RPipes(v_0, v(\Omega))$ then $LD(\pi(P)) \in W(\Omega)$. Hence we have the inclusion
\begin{align*}
\bigcup_{\mathbf{w} \in W(\Omega)} \RPipeNet(\mathbf{w}) \supseteq \RPipes(v_0, v(\Omega)).
\end{align*}
The reverse inclusion is the content of Lemma \ref{lem:PNtoRedPipes}.
\end{proof}

Before we prove the main theorem of this section, we introduce some notation that will be helpful below. Fix a lacing diagram $\mathbf{w}$. Given any subset $S=\{\beta_{i_1}, \dots, \beta_{i_p}, \alpha_{j_1}, \dots, \alpha_{j_q}\}$ of blocks of the snake region let $\PipeNet(\mathbf{w}; S)$ denote the set of equivalence classes of pipe dreams in $\PipeNet(\mathbf{w})$ taken with respect to the following equivalence relation: $P\sim_S P'$ if and only if $P$ and $P'$ are equal when restricted to any snake region block \textit{not} in $S$. Explicitly, we have 
\begin{align*}
P \sim_S P' \Leftrightarrow P \setminus \{P_{i_1}, \dots, P_{i_p}, P^{j_1}, \dots, P^{j_q}\} = P' \setminus \{ P'_{i_1}, \dots, P'_{i_p}, (P')^{j_1}, \dots, (P')^{j_q}\}.
\end{align*}
We similarly write $\RPipeNet(\mathbf{w}; S)$ for the analogous set of equivalences classes in $\RPipeNet(\mathbf{w})$.

\begin{remark}\label{rem:equivSums}
As a final remark, below we will write formulae involving sums over these sets of equivalence classes; i.e. expressions of the form $\sum_{[P] \in \RPipeNet(\mathbf{w}; S)} f(P)$ for some function $f$. This is to be interpreted as the evaluation of $f$ on any representative of the equivalence class $[P]$. Whenever we use this slight abuse of notation the function $f$ will be independent of the choice of representative of $[P]$. 
\end{remark}

We now complete our combinatorial proof of the cohomological  bipartite type $A$ quiver component formula, first proved by Buch and Rim\'anyi in \cite{BuchRim} (see also \cite{KKR}).

\begin{theorem}\label{thm:3.9}
For any bipartite type $A$ quiver locus $\Omega$, its multidegree can be written as
\begin{align}
\mathcal{Q}_{\Omega}(\mathbf{t}-\mathbf{s})&=\sum_{\mathbf{w} \in W(\Omega)} \mathfrak{S}_\mathbf{w}(\mathbf{t};\mathbf{s}) \,.\label{eqn:multidegree}
\end{align}
\end{theorem}

\begin{proof}
First, note that by the multidegree version of the bipartite pipe formula (see Theorem \ref{thm:pipeformulas}) and Proposition \ref{lem:3.7}, there is an equality
\begin{align*}
\mathcal{Q}_{\Omega}(\mathbf{t}-\mathbf{s})=\sum_{P \in \RPipes(v_0, v(\Omega))} (\mathbf{t}-\mathbf{s})^{P \setminus P_*} =\sum_{\mathbf{w} \in W(\Omega)} \sum_{P \in \RPipeNet(\mathbf{w})} (\mathbf{t}-\mathbf{s})^{P\setminus P_*}\,.
\end{align*}
To prove \eqref{eqn:multidegree} it is therefore sufficient to show that for each fixed $\mathbf{w} \in W(\Omega)$ we have
\begin{align}\label{eqn:pipesEquiv}
\sum_{P \in \RPipeNet(\mathbf{w})} (\mathbf{t}-\mathbf{s})^{P\setminus P_*} &=\prod_{i=1}^n \mathfrak{S}_{w_i}(\mathbf{t}^i;\mathbf{s}^i) \prod_{i=1}^n \mathfrak{S}_{\rot (w^i)} (\tilde{\mathbf{t}}^{i-1}; \tilde{\mathbf{s}}^{i}) \,.
\end{align}
We may then rewrite the left-hand-side of \eqref{eqn:pipesEquiv} as
\begin{align}\label{eqn:equivClasses}
\sum_{[P] \in \RPipeNet(\mathbf{w}; \beta_n)} \sum_{P \in [P]} \prod_{i=1}^n (\mathbf{t}^i-\mathbf{s}^i)^{P_i}\prod_{i=1}^n (\tilde{\mathbf{t}}^{i-1} -\tilde{\mathbf{s}}^{i})^{\rot(P^i)}\,.
\end{align}
Fix an equivalence class $[P] \in \RPipeNet(\bw;\beta_n)$ and consider the sum 
\begin{align}\label{eqn:betansum}
    \sum_{P \in [P]} \prod_{i=1}^n (\bt^i - \bs^i)^{P_i} \prod_{i=1}^n (\tilde{\bt}^{i-1} - \tilde{\bs}^i)^{\rot(P^i)}.
\end{align}
Since every element $P \in [P]$ has the same configuration of cross and elbow tiles outside of the $\beta_n$ block, every term in Equation \eqref{eqn:betansum} has the common factor 
\begin{align}\label{eqn:commonfactor}
    \prod_{i=i}^{n-1} (\bt^i-\bs^i)^{P_i}\prod_{i=1}^n(\tilde{\bt}^{i-1} - \tilde{\bs}^i)^{\rot(P^i)}
\end{align}
where the $P_i$ and $P^i$ in this expression are the mini pipes dreams of an arbitrary element $P \in [P]$. Upon factoring Equation \eqref{eqn:commonfactor} from every term in Equation \eqref{eqn:betansum}, we may rewrite Equation \eqref{eqn:betansum} as 
\begin{align*}
    \left(\sum_{P_n \in \RPipes(w_n)} (\bt^n - \bs^n)^{P_n} \right) \left(   \prod_{i=i}^{n-1} (\bt^i-\bs^i)^{P_i} \prod_{i=1}^n(\tilde{\bt}^{i-1} - \tilde{\bs}^i)^{\rot(P^i)} \right).
\end{align*}
Indeed, by construction of $\RPipeNet(\mathbf{w})$, for each $P_n \in \RPipes(w_n)$ there is a unique representative of the equivalence class $[P]$ having $P_n$ in its $\beta_n$ block. Noting that the factor $\sum_{P_n \in \RPipes(w_n)} (\bt^n-\bs^n)^{P_n}$ is independent of the equivalence class $[P]$ we conclude that we may rewrite Equation \eqref{eqn:equivClasses} as
\begin{align*}
\left( \sum_{P_n \in \RPipes(w_n)} (\mathbf{t}^n-\mathbf{s}^n)^{P_n} \right)\cdot \left( \sum_{[P] \in \RPipeNet(\mathbf{w}; \beta_n)} \prod_{i=1}^{n-1} (\mathbf{t}^i-\mathbf{s}^i)^{P_i}\prod_{i=1}^n (\tilde{\mathbf{t}}^{i-1} -\tilde{\mathbf{s}}^{i})^{\rot(P^i)}\right),
\end{align*}
where the sum over $[P]$ is to be understood as in Remark \ref{rem:equivSums}. Noting that the first term in this product is preciesly $\mathfrak{S}_{w_n}(\mathbf{t};\mathbf{s})$, we can rewrite this equation as 
\begin{align*}
\mathfrak{S}_{w_n}(\mathbf{t}^n;\mathbf{s}^n) \sum_{[P] \in \RPipeNet(\mathbf{w}; \beta_n)}  \prod_{i=1}^{n-1} (\mathbf{t}^i-\mathbf{s}^i)^{P_i}\prod_{i=1}^n (\tilde{\mathbf{t}}^{i-1} -\tilde{\mathbf{s}}^{i})^{\rot(P^i)}.
\end{align*}
By breaking the sum over $\RPipeNet(\mathbf{w}; \beta_n)$ into a sum over equivalence classes in \newline $\RPipeNet(\mathbf{w}; \beta_n, \alpha_n)$ and representatives thereof (c.f. equation \eqref{eqn:equivClasses}) we can similarly factorize $\mathfrak{S}_{\text{rot}(w^n)}(\tilde{\mathbf{t}}^{n-1}; \tilde{\mathbf{s}}^{n})$ from this expression. Continuing in this fashion, we obtain \eqref{eqn:multidegree}.
\end{proof}

\section{A combinatorial proof of the bipartite $K$-theoretic component formula}

In this section we give a direct combinatorial proof of the bipartite $K$-theoretic quiver component formula, originally proven in \cite{KKR}. To do this, we will identify the set $KW(\Omega)$ of $K$-theoretic lacing diagrams for a bipartite type $A$ quiver locus $\Omega$ with a subset $X_\Omega$ of $S_\mathbf{d}$. This will allow us to factor the formula for $KQ_\Omega(\mathbf{t}/\mathbf{s})$ in a manner similar to the factorization of $\mathcal{Q}_\Omega(\mathbf{t}-\mathbf{s})$ in the proof of Theorem \ref{thm:3.9}.

\subsection{A bijection on sets of pipe dreams}

The main goal of this subsection is to prove Proposition \ref{prop:pipeDreamBijection}, which gives a bijection between $\Pipes(v_0,v(\Omega))$ and a particular product of sets of pipe dreams. We begin with some definitions and notation. Define the following subset of $S_{\bd}$: 
\begin{equation}
X_\Omega:= \{\pi(P)\mid P\in \Pipes(v_0,v(\Omega))\}.
\end{equation}

\begin{lemma}\label{prop:goodPipes}
Let $v \in S_d$ and $P\in \Pipes(v_0,v)$, with $\pi(P)\in X_{\Omega}$. Then 
$\delta(P) = v(\Omega)$, that is, $P\in \Pipes(v_0,v(\Omega))$. 
\end{lemma}

\begin{proof}
Our proof is essentially the same as \cite[Proposition 7]{Miller}. Let $T^k$ be the pipe dream on a $d_y\times d_x$ grid which has cross tiles in all positions strictly north of the $\alpha_k$ block of the snake region and elbows elsewhere, and let $T_k$ be the pipe dream on a $d_y\times d_x$ grid which has cross tiles in all positions strictly south of the $\beta_k$ block of the snake region and elbows elsewhere. Let $\varepsilon^k = \delta(T^k)$ and $\varepsilon_k = \delta(T_k)$.

Let $P\in \Pipes(v_0, v)$. Let $\tilde{P}^k$ be the sub pipe dream on a $d_y\times d_x$ grid which is equal to $P$ on the $\alpha_k$ block of the snake region and has elbow tiles in all other positions, and let $\tilde{P}_k$ be the sub pipe dream on a $d_y\times d_x$ grid which is equal to $P$ on the $\beta_k$ block of the snake region and has elbows in all other positions.\footnote{Note that $\tilde{P}^k$ and $\tilde{P}_k$ are different from the mini pipe dreams $P^k$ and $P_k$ obtained from $P$ (see Section \ref{sec:pipestolaces}) since $\tilde{P}^k$ and $\tilde{P}_k$ are pipe dreams on the full $d_y\times d_x$ grid.} Let $w^k = \delta(\tilde{P}^k)$ and let $w_k = \delta(\tilde{P}_k)$. 
We will first show that $\delta(P)$ is equal to the Demazure product
  \begin{equation}\label{eq:factorization}
  (w^1w_1\varepsilon_1)\left[(\varepsilon^2w^2w_2\varepsilon_2)\cdots (\varepsilon^k w^{k}w_{k}\varepsilon_{k})\cdots (\varepsilon^{n-1}w^{n-1}w_{n-1}\varepsilon_{n-1})\right](\varepsilon^n w^nw_n).
\end{equation}

Let $P(x_k)$ be the sub pipe dream of $P$ which agrees with $P$ on block column $x_k$ and has elbow tiles elsewhere. 
By results from \cite{FK1} (see also \cite[Proposition 7]{Miller}), $\delta(P)$ may be computed by first reading from top to bottom down the rightmost column, then from top to bottom down the second to rightmost column, then continuing to move to the west in this fashion. Consequently, we see that
\[
\delta(P) = \delta(P(x_1))\delta(P(x_2))\cdots \delta(P(x_n)).
\]
Then, using the ordinary reading order for Demazure product, we see that
\begin{equation}
\delta(P(x_1)) = w^1w_1\varepsilon_1, \quad \delta(P(x_k)) = \varepsilon^k w^{k}w_{k}\varepsilon_{k}, \text{ for }2\leq k\leq n-1, \quad \delta(P(x_n)) = \varepsilon^n w^nw_n,
\end{equation}
and thus \eqref{eq:factorization} follows.

Now, if $\pi(P)\in X_{\Omega}$, then there is some $Q\in \Pipes(v_0,v(\Omega))$ satisfying $\pi(Q) = \pi(P)$. In other words, $\delta(P_k) = \delta(Q_k)$ for all $k$ and $\delta(\text{rot}(P^k)) = \delta(\text{rot}(Q^k))$ for all $k$, which implies that $\delta(P^k) = \delta(Q^k)$. It follows that $\delta(\tilde{P}_k) = \delta(\tilde{Q}_k)$ and $\delta(\tilde{P}^k) = \delta(\tilde{Q}^k)$, and so the lemma follows  from \eqref{eq:factorization}. 
\end{proof}

For $\mathbf{v}=\pi(P) \in X_\Omega$ recall that we have $v_i =\delta(P_i)$ and $v^i=\delta(\rot(P^i))$. We are now ready to prove the main result of this subsection.

\begin{proposition}\label{prop:pipeDreamBijection}
The pipe network map
\begin{equation}\label{eq:PN}
\bigcup_{\bv\in X_\Omega} \left(\prod_{i=1}^n (\Pipes(v_i)\times \Pipes(v^i))\right) \xrightarrow{PN} \Pipes(v_0, v(\Omega))
\end{equation}
sending a tuple of pipe dreams $(P_n, P^n, \dots, P_1, P^1)$ to the pipe dream in $\Pipes(v_0, v(\Omega))$ having $P_k$ in its $\beta_k$ block and $\text{rot}(P^k)$ in its $\alpha_k$ block for each $k$ is a bijection.
\end{proposition}

\begin{proof}
Note that by Corollary \ref{lem:partialperm}, we have that each $v_i$ is the completion of some $\bd(y_i)\times\bd(x_i)$ partial permutation, and each $v^i$ is the completion of a rotation of a $\bd(y_{i-1}) \times \bd(x_i)$ partial permutation. Thus, by Lemma \ref{lem:completedPipes}, every $P \in \Pipes(v_i)$ can be written on a $\bd(y_i)\times \bd(x_i)$ grid, and every $P \in \Pipes(v^i)$ can be written on a $\bd(y_{i-1})\times \bd(x_i)$ grid. Therefore the pipe dreams $(P_n, P^n, \dots, P_1, P^1)$ are of the correct sizes to be embedded into the snake region. Together with Lemma \ref{prop:goodPipes}, we see that the map $PN$ takes the domain of \eqref{eq:PN} to $\Pipes(v_0,v(\Omega))$, and therefore this map is well defined. 
To see that the map $PN$ is surjective, take any $P\in \Pipes(v_0, v(\Omega))$ and extract its mini pipe dreams $(P_n,P^n,\dots, P_1,P^1)$. Then the tuple $(P_n, \text{rot}(P^n),\dots, P_1,\text{rot}(P^1))$ is in the domain of the $PN$ map since $\pi(P) \in X_\Omega$ by definition of $X_\Omega$, and furthermore, its pipe network is $P$. Thus $PN$ is surjective.
Injectivity is clear by construction.
\end{proof}

We end this subsection with a characterization of $\bv\in X_{\Omega}$ in terms of factorizations of the Zelevinsky permutation $v(\Omega)$, in analogy with KMS factorizations from the equioriented setting (see \cite[\S6]{MR2114821}). For $1\leq k\leq n$, let $a_k = \sum_{i> k}\bd(x_i)+\sum_{i< k-1}\bd(y_i)$ and $b_k = \sum_{i>k}\bd(x_i)+\sum_{i<k}\bd(y_i)$.
Observe that the northwest corner of the $\alpha_k$ block of the snake region is on the $(a_k+1)^{\rm{th}}$ diagonal of the $d_y\times d_x$ grid and the northwest corner of the $\beta_k$ block is on the $(b_k+1)^{\rm{th}}$ diagonal. If $v$ is a permutation in $S_n$ and $m$ is a positive integer, let $1^m\times v\in S_{n+m}$ be the permutation satisfying $(1^m\times v)(i) = i$, $i\leq m$, and $(1^m\times v)(i) = v(i-m)+m$, $i>m$. 
Finally, let $R^k$ be the pipe dream on a $d_y\times d_x$ grid which has cross tiles in all positions strictly west of the $\alpha_k$ block of the snake region and elbows elsewhere, and let $R_k$ be the pipe dream on a $d_y\times d_x$ grid which has cross tiles in all positions strictly east of the $\beta_k$ block of the snake region. Let $\theta^k = \delta(R^k)$ and $\theta_k = \delta(R_k)$.

\begin{lemma}\label{lem:KMSfactorization}
Let $\bv = (v_n,v^n,\dots, v_1,v^1)\in S_{\bd}$, let $\varepsilon^k, \varepsilon_k$ be as defined in the proof of Lemma \ref{prop:goodPipes}, let $a_k, b_k, \theta^k, \theta_k$ be as defined above, and let $w_k = 1^{b_k} \times v_k$ and $w^k = 1^{a_k}\times \text{rot}((v^k)^{-1}) $. Then the following are equivalent:
\begin{enumerate}
\item $\bv\in X_{\Omega}$. 
\item $v(\Omega)$ is equal to the Demazure product
  \begin{equation}\label{eq:factorization2}
  (w^1w_1\varepsilon_1)\left[(\varepsilon^2w^2w_2\varepsilon_2)\cdots (\varepsilon^k w^{k}w_{k}\varepsilon_{k})\cdots (\varepsilon^{n-1}w^{n-1}w_{n-1}\varepsilon_{n-1})\right](\varepsilon^n w^nw_n).
\end{equation}
\item $v(\Omega)$ is equal to the Demazure product 
\begin{equation}\label{eq:factorization3}
(w^1\theta^1)(w_1w^2\theta^2)\left[(\theta_2 w_2w^{3}\theta^3)\cdots (\theta_{n-2} w_{n-2}w^{n-1}\theta^{n-1})\right](\theta_{n-1}w_{n-1}w^n)(\theta_nw_n).
\end{equation}
\end{enumerate}
\end{lemma}

\begin{proof}
We prove that (1) and (2) are equivalent. The proof that (1) and (3) are equivalent is analogous. First suppose that $\bv\in X_\Omega$. Let $P\in \Pipes(v_0,v(\Omega))$ with $\pi(P) = \bv$. 
By the argument given in the proof of Lemma \ref{prop:goodPipes}, we see that $v(\Omega)$ is equal to the Demazure product \eqref{eq:factorization2} so long as $\delta(P_k) = v_k\in S_{\beta_k}$ and $\delta(P^k) = \text{rot}((v^k)^{-1})\in S_{\alpha_k}$. But this follows from the assumption $\pi(P) = \bv$ together with the fact that if $\delta(\text{rot}(P^k)) = v^k\in S_{\alpha_k}$, then $\delta(P^k) = \text{rot}((v^k)^{-1})$. For the converse, let $Q  = (Q_n,Q^n,\dots, Q_1,Q^1)$ with $Q^k\in \Pipes(v^k)$, $Q_k\in \Pipes(v_k)$. Then $\text{rot}(Q^k)$ is a pipe dream for $\text{rot}((v^k)^{-1})$. So, again by the argument from the proof of Lemma \ref{prop:goodPipes}, $\delta(PN(Q))$ is equal to the Demazure product from \eqref{eq:factorization2}, which by assumption is $v(\Omega)$. In other words, $PN(Q)\in \Pipes(v_0,v(\Omega))$ and so $\bv = \pi(PN(Q))\in X_\Omega$ by definition of $X_{\Omega}$. 
\end{proof}

\subsection{Bijection from $X_{\Omega}$ to $K$-theoretic lacing diagrams}

By Lemma \ref{lem:LDPerm} we have the diagram 
\begin{center}
    \begin{tikzcd}
        \, & X_\Omega \ar[r, hook] \ar[d, dotted, "LD"] & S_{\bd} \ar[d, "LD", hook] \\
       W(\Omega) \ar[r, hook] & KW(\Omega) \ar[r, hook] & \Laces(\bd) & \, 
    \end{tikzcd}
\end{center}
In this subsection we prove that $LD$ restricts to a bijection from $X_\Omega$ to $KW(\Omega)$, the collection of $K$-theoretic lacing diagrams for $\Omega$; i.e. the dotted arrow in the diagram is a bijection.

The following is the bipartite analogue of the remark at the end of \cite[\S6]{MR2114821}. This lemma interprets the $K$-theoretic lacing diagram moves shown below in terms of the associated elements in $X_\Omega$. Recall that, for these $K$-theoretic moves, both of the dots in the middle column are non-virtual and are adjacent in their column, and at least one dot in each of the outer columns is non-virtual.

\begin{center}
\begin{tikzpicture}
\node (1) at (-3,1) {$\bullet$}; 
\node (2) at (-3,0) {$\bullet$};
\node (3) at (-2,1) {$\bullet$};
\node (4) at (-2,0) {$\bullet$};
\node (5) at (-1,1) {$\bullet$};
\node (6) at (-1,0) {$\bullet$};

\node (arrow) at (0,.5) {$\longleftrightarrow$};

\node (66) at (3,1) {$\bullet$}; 
\node (55) at (3,0) {$\bullet$};
\node (44) at (2,1) {$\bullet$};
\node (33) at (2,0) {$\bullet$};
\node (22) at (1,1) {$\bullet$};
\node (11) at (1,0) {$\bullet$};

\node (arrow2) at (4, .5) {$\longleftrightarrow$};

\node (666) at (7,1) {$\bullet$}; 
\node (555) at (7,0) {$\bullet$};
\node (444) at (6,1) {$\bullet$};
\node (333) at (6,0) {$\bullet$};
\node (222) at (5,1) {$\bullet$};
\node (111) at (5,0) {$\bullet$};

\draw[-, thick] (1) -- (4);
\draw[-, thick] (2) -- (3);
\draw[-, thick] (3) -- (5);
\draw[-, thick] (4) -- (6);

\draw[-, thick] (11) -- (44);
\draw[-, thick] (22) -- (33);
\draw[-, thick] (33) -- (66);
\draw[-, thick] (44) -- (55);

\draw[-, thick] (111) -- (333);
\draw[-, thick] (222) -- (444);
\draw[-, thick] (333) -- (666);
\draw[-, thick] (444) -- (555);
\end{tikzpicture}
\end{center}

\begin{lemma} \label{lem:KTheoreticPermutations}
Let $\bv = (v_n,v^n,\dots, v_1,v^1)\in S_{\bd}$. 
\begin{enumerate}
\item Suppose $1\leq i\leq n$ and $1\leq k< \bd(x_i)$. Let $k' = \bd(y_{i-1})+k$, and suppose that $v_i^{-1}(k)<v_i^{-1}(k+1)$ and $(\rot(v^i))^{-1}(k')<(\rot(v^i))^{-1}(k'+1)$. Then, if one of the following is in $X_{\Omega}$, all three are in $X_{\Omega}$: 
	\begin{enumerate}
	\item $(v_n,v^n,\dots, \tau_kv_i , v^i,\dots, v_1,v^1),$
	\item $(v_n,v^n,\dots, v_i, \rot(\tau_{k'}\rot(v^i)),\dots, v_1,v^1),$
	\item $(v_n,v^n,\dots,\tau_k v_i, \rot(\tau_{k'}\rot(v^i)),\dots, v_1,v^1)$.
	\end{enumerate}
\item Suppose $1\leq i\leq n-1$ and $1\leq l< \bd(y_i)$. Let $l' = \bd(x_{i+1})+l$, and suppose that $v_i(l)<v_i(l+1)$ and $v^{i+1}(l')<v^{i+1}(l'+1)$. Then, if one of the following is in $X_{\Omega}$, all three are in $X_{\Omega}$: 
	\begin{enumerate}
	\item $(v_n,v^n,\dots v^{i+1}, v_i \tau_l,\dots, v_1,v^1)$,
	\item $(v_n,v^n,\dots, \rot(\rot(v^{i+1})\tau_{l'}), v_i,\dots, v_1,v^1)$, 
	\item $(v_n,v^n,\dots, \rot(\rot(v^{i+1})\tau_{l'}), v_i \tau_l,\dots, v_1,v^1)$. 
	\end{enumerate}
    \item Let $(v_n, v^n, \dots, v_1, v^1) \in X_{\Omega}$. Then, $\rot(v^1)(k')<\rot(v^1)(k'+1)$ for every $d(x_1)+1 \leq k' < d(x_1)+d(y_0)$, and $v_n(j) < v_n(j+1)$ for every $1 \leq j < d(y_n)$.
\end{enumerate}
\end{lemma}

\begin{proof}
We use the notation introduced in Lemma \ref{lem:KMSfactorization}. 
By Lemma \ref{lem:KMSfactorization}, $\bv\in X_{\Omega}$ if and only if \eqref{eq:factorization2} holds. By associativity of Demazure product, to prove (1), it suffices to observe that the following three Demazure products agree:
\begin{itemize}
\item $(1^{a_i}\times w_0(v^i)^{-1}w_0)(1^{b_i} \times \tau_kv_i)$,  
\item $(1^{a_i}\times w_0(w_0\tau_{k'}(w_0v^iw_0)w_0)^{-1}w_0)(1^{b_i} \times v_i)$, and
\item $(1^{a_i}\times w_0(w_0\tau_{k'}(w_0v^iw_0)w_0)^{-1}w_0)(1^{b_i} \times \tau_kv_i)$, 
\end{itemize}where $w_0$ is the longest permutation in $S_{\bd(y_i)+\bd(x_i)}$. (That is, the permutation matrix of $w_0$ has $1$s along the antidiagonal and $0$s elsewhere.) 
But this follows by noting that we have the following equality of permutations
\[
1^{a_i}\times w_0(w_0\tau_{k'}(w_0v^iw_0)w_0)^{-1}w_0 = 1^{a_i}\times (w_0(v^i)^{-1}w_0\tau_{k'}) = (1^{a_i}\times w_0(v^i)^{-1}w_0)\tau_{k''}
\]
where $k'' = k'+a_i = k+b_i$, and that we also have $1^{b_i} \times \tau_kv_i = \tau_{k''}(1^{b_i}\times v_i)$. 

Similarly, to prove (2), we use that $\bv\in X_{\Omega}$ if and only if \eqref{eq:factorization3} holds. So, it suffices to observe the following three Demazure products agree:  
\begin{itemize}
    \item $(1^{b_i}\times v_i\tau_l)(1^{a_{i+1}}\times w_0(v^{i+1})^{-1}w_0)$,
    \item $(1^{b_i}\times v_i)(1^{a_{i+1}}\times w_0(w_0(w_0v^{i+1}w_0)\tau_{l'}w_0)^{-1}w_0)$, and  
    \item $(1^{b_i}\times v_i\tau_l)(1^{a_{i+1}}\times w_0(w_0(w_0v^{i+1}w_0)\tau_{l'}w_0)^{-1}w_0)$.
    \end{itemize}
    
This follows since
\[
1^{a_{i+1}}\times w_0(w_0(w_0v^{i+1}w_0)\tau_{l'}w_0)^{-1}w_0) = 1^{a_{i+1}}\times(\tau_{l'}w_0v^{i+1}w_0) = \tau_{l''}(1^{a_{i+1}}\times (w_0(v^{i+1})^{-1}w_0),
\]
where $l'' = l'+a_{i+1} = l+b_i$, and since $1^{b_i}\times (v_i\tau_l) = (1^{b_i} \times v_i)\tau_{l''}$. 

Finally, we prove part $(3)$. 
Let $\bv\in X_\Omega$ and let $P \in \Pipes(v_0, v(\Omega))$ be any pipe dream with $\pi(P)= \bv$. Since $P$ is a pipe dream for the Zelevinsky permutation $v(\Omega)$, and $1$s appear northwest to southeast across block rows and down block columns in $v(\Omega)$, two pipes which enter the $\beta_n$ block of the snake region on the left of the block (i.e., in the block row labeled by $y_n$), do not cross. Since $v_n = \delta(P_n)$, this implies that $v_n(j) < v_n(j+1)$ for every $1 \leq j < d(y_n)$. Similarly, two pipes which exit the $\alpha_1$ block at the right of the block (i.e., in the $y_0$ block row) do not cross. Since $v^1 = \delta(\rot(P^1))$, this implies that $(v^1)(k') < v^1(k'+1)$ for every $1 \leq k < d(y_0)$, which implies that $\rot(v^1)(k')<\rot(v^1)(k'+1)$ for every $d(x_1)+1 \leq k' < d(x_1)+d(y_0)$.
\end{proof}

\begin{remark}\label{rem:KMSconditions}
    On \cite[Page 318]{BFR}, two conditions on equioriented lacing diagrams are stated. The authors then prove that the collection KMS-factorizations for an equioriented type A-quiver locus satisfy these two conditions. Parts $(1)$ and $(2)$ of Lemma \ref{lem:KTheoreticPermutations} show that $X_\Omega$, our bipartite version of KMS-factorizations, satisfies the bipartite analogue of condition $(\text{II})$ from \cite[Page 318]{BFR}, while part $(3)$ shows that elements of $X_{\Omega}$ satisfy the bipartite analogue of condition $(\text{I})$ in \emph{loc. cit.}. 
\end{remark}

We will refer to an operation interchanging any of $1(a), 1(b),$ or $1(c)$ in Lemma \ref{lem:KTheoreticPermutations} as a ``move'' applied to an element of $S_\bd$ satisfying the assumptions of Lemma \ref{lem:KTheoreticPermutations} (1). We similarly refer to an operation interchanging any of the permutations in $2(a), 2(b),$ or $2(c)$ as a ``move'' applied to an element of $S_\bd$ satisfying the assumptions of Lemma \ref{lem:KTheoreticPermutations} (2). We will next see that performing these moves is equivalent to performing a $K$-theoretic lace diagram move to a corresponding completed lacing diagram.

\begin{lemma}\label{lem:KThEquivariance}
    Let $\bw, \bw' \in KW(\Omega)$. Then, the extended lacing diagram of $\bw'$ can be obtained from the extended lacing diagram of $\bw$ via a single $K$-theoretic lacing diagram move if and only if $\fc(\bw)$ and $\fc(\bw')$ are related by one of the moves listed in Lemma \ref{lem:KTheoreticPermutations}.
\end{lemma}

\begin{proof}
    Let $\bv=\fc(\bw)$ and $\bv'=\fc(\bw')$. Let us assume that $\bw$ and $\bw'$ are related by the first lacing diagram move in the figure above Lemma \ref{lem:KTheoreticPermutations} and that the leftmost laces in each diagram correspond to right pointing arrows in $Q$, while the rightmost laces in each diagram correspond to left pointing arrows in $Q$. Then, $\bv$ can be written in the form appearing in point $1(a)$ in Lemma \ref{lem:KTheoreticPermutations}, and $\bv'$ is the permutation appearing in $1(c)$. Thus, this $K$-theoretic lacing diagram move can be written permutation-theoretically via the move sending a permutation of the form $1(a)$ to the permutation appearing in $1(c)$. The other cases are similar. 
    
    By reversing the above argument, we easily see that if $\bv, \bv' \in S_{\bd}$ are related by one of the moves in \ref{lem:KTheoreticPermutations}, then $LD(\bv)$ and $LD(\bv')$ are related by the corresponding $K$-theoretic move.
\end{proof}

\begin{example}
   In this example we will see how to translate between lacing diagrams and elements of $S_{\bd}$ and see how moves in Lemma \ref{lem:KTheoreticPermutations} correspond to K-theoretic lacing diagram moves. Consider the following lacing diagram $\bw$, drawn on the left below, and its extended diagram, drawn on the right:
    
    
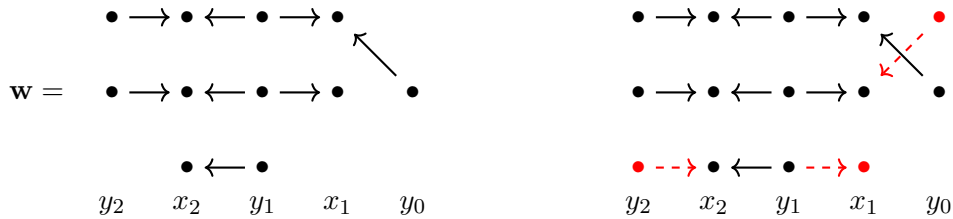
\begin{center}
    \begin{tikzpicture}
\node (w11) at (0,2) {$\bullet$};
\node (w12) at (0,1) {$\bullet$};

\node (w21) at (1,2) {$\bullet$};
\node (w22) at (1,1) {$\bullet$};
\node (w23) at (1,0) {$\bullet$};

\node (w31) at (2,2) {$\bullet$};
\node (w32) at (2,1) {$\bullet$};
\node (w33) at (2,0) {$\bullet$};

\node (w41) at (3,2) {$\bullet$};
\node (w42) at (3,1) {$\bullet$};

\node (w51) at (4,1) {$\bullet$};

\node (y2) at (0,-.5) {$y_2$};
\node (x2) at (1,-.5) {$x_2$};
\node (y1) at (2,-.5) {$y_1$};
\node (x1) at (3,-.5) {$x_1$};
\node (y0) at (4,-.5) {$y_0$};

\node (w) at (-1, 1) {$\mathbf{w}=$};

\draw[->, thick] (w11) -- (w21);
\draw[->, thick] (w12) -- (w22);

\draw[->, thick] (w31) -- (w21);
\draw[->, thick] (w32) -- (w22);
\draw[->, thick] (w33) -- (w23);

\draw[->, thick] (w31) -- (w41);
\draw[->, thick] (w32) -- (w42);

\draw[->, thick] (w51) -- (w41);

\node (cw11) at (7,2) {$\bullet$};
\node (cw12) at (7,1) {$\bullet$};
\node[red] (cw13) at (7,0) {$\bullet$};

\node (cw21) at (8,2) {$\bullet$};
\node (cw22) at (8,1) {$\bullet$};
\node (cw23) at (8,0) {$\bullet$};

\node (cw31) at (9,2) {$\bullet$};
\node (cw32) at (9,1) {$\bullet$};
\node (cw33) at (9,0) {$\bullet$};

\node (cw41) at (10,2) {$\bullet$};
\node (cw42) at (10,1) {$\bullet$};
\node[red] (cw43) at (10,0) {$\bullet$};

\node (cw51) at (11,1) {$\bullet$};
\node[red] (cw50) at (11,2) {$\bullet$};

\node (cy2) at (7,-.5) {$y_2$};
\node (cx2) at (8,-.5) {$x_2$};
\node (cy1) at (9,-.5) {$y_1$};
\node (cx1) at (10,-.5) {$x_1$};
\node (cy0) at (11,-.5) {$y_0$};

\draw[->, thick] (cw11) -- (cw21);
\draw[->, thick] (cw12) -- (cw22);
\draw[->, thick, dashed, red] (cw13) -- (cw23);

\draw[->, thick] (cw31) -- (cw21);
\draw[->, thick] (cw32) -- (cw22);
\draw[->, thick] (cw33) -- (cw23);

\draw[->, thick] (cw31) -- (cw41);
\draw[->, thick] (cw32) -- (cw42);
\draw[->, thick, dashed, red] (cw33) -- (cw43);

\draw[->, thick] (cw51) -- (cw41);
\draw[->, thick, dashed, red] (cw50) -- (cw42);
\end{tikzpicture}
\end{center}

 Note that in the extended lacing diagram above, we are drawing the \textit{minimal} extension of $\bw$. However, Lemma \ref{lem:KTheoreticPermutations} is stated for sequences of permutations in $S_{\bd}$. That is, we are using the map from \eqref{eqn:permOperator} to identify lacing diagrams with elements of $S_\bd$. So, for example, we should interpret the arrows from $y_1$ to $x_2$ as defining the identity permutation on \textit{six} elements, not three, since $\bd(y_1)+\bd(x_2) = 3+3 = 6$. Let $\bv=\fc(\bw) \in S_{\bd}$.  

 Now let $i = 2$ and $k = 2$ in Lemma \ref{lem:KTheoreticPermutations} (1). Then $k' = \bd(y_1)+2 = 3+2 = 5$. Then, $v_i^{-1}(k)<v_i^{-1}(k+1)$ and $(\rot(v^i))^{-1}(k')<(\rot(v^i))^{-1}(k'+1)$. This is equivalent to noticing that the bottom two laces between the $y_2$ and $x_2$ column in the extended diagram don't cross and that the bottom two laces between the $x_2$ and $y_1$ columns don't cross. 

The element of $S_\bd$ associated to the lacing diagram in Figure \ref{fig:laces} then has the form
\[
\bv' = (v_2, \rot(\tau_5 \rot(v^2)),v_1,v^1)
\]
and the element of $S_\bd$ associated to the lacing diagram on the left in Example \ref{ex:KThDiagrams} has the form
\[
\bv'' = (\tau_2v_2, \rot(\tau_5 \rot(v^2)),v_1,v^1).
\]
Notice that $\bv'$ and $\bv''$ have the form of $1(b)$ and $1(c)$ in Lemma \ref{lem:KTheoreticPermutations}, while the associated extended lacing diagrams are related by a $K$-theoretic lace move.
 \end{example}

We are now in a position to prove our desired bijection between $X_\Omega$ and $KW(\Omega)$. Let $X_\Omega^{red} \subset X_{\Omega}$ denote the set of $\bv \in X_\Omega$ such that $\bv=\pi(P)$ for some \emph{reduced} pipe dream $P \in \RPipes(v_0, v(\Omega))$. The following is the bipartite analogue of \cite[Theorem 3]{BFR}.

\begin{proposition}\label{prop:KMSsurjection}
    Every $\bv \in X_\Omega$ can be obtained from an element of $X_\Omega^{red}$ by applying a sequence of the moves listed in Lemma \ref{lem:KTheoreticPermutations}.
\end{proposition}

\begin{proof}
    The proof of \cite[Theorem 3]{BFR} carries over verbatim to the bipartite setting, using the fact that elements of $X_{\Omega}$ satisfy the bipartite analogues of conditions $(\text{I})$ and $(\text{II})$ from \cite[Page 318]{BFR}, as discussed in Remark \ref{rem:KMSconditions}.
\end{proof}

\begin{corollary}\label{cor:LDonto}
    If $\bw=LD(\bv)$ is the lacing diagram associated to some $\bv \in X_{\Omega}$, then $\bw \in KW(\Omega)$. 
\end{corollary}

\begin{proof}
    By Proposition \ref{prop:KMSsurjection}, $\bv$ can be obtained from some $\bv' \in X_{\Omega}^{red}$ by a sequence of moves as in Lemma \ref{lem:KTheoreticPermutations}. Let $P \in \RPipes(v_0, v(\Omega))$ be any reduced pipe dream such that $\pi(P)=\bv'$. By Proposition \ref{prop:followpipes} we have that $\pi(P)=\fc(\fw(P))$ and by Proposition \ref{prop:wPminLacingDiagram}, $\fw(P)$ is a minimal lacing diagram. Since the moves in Lemma \ref{lem:KTheoreticPermutations} correspond to $K$-theoretic lacing diagram moves by Lemma \ref{lem:KThEquivariance}, it follows that the lacing diagram $\bw=LD(\bv)$ associated to $\bv$ can be obtained by applying $K$-theoretic lacing diagram moves to the minimal lacing diagram $\fw(P)$. Therefore $\bw \in KW(\Omega)$. 
\end{proof}

\begin{proposition}\label{prop:lacesBijection}
The assignment $\bv\mapsto LD(\bv)$ induces a bijection $X_\Omega \overset{\sim}{\to}KW(\Omega)$.
\end{proposition}

\begin{proof}
    Injectivity of this map is obvious from construction. By Corollary \ref{cor:LDonto} its image is a subset of $KW(\Omega)$. To see surjectivity, recall that $KW(\Omega)$ was defined to be the set of lacing diagrams obtained by applying $K$-theoretic lacing diagram moves to (the completions of) elements of $W(\Omega)$. Thus, let $\bw \in KW(\Omega)$ be a $K$-theoretic lacing diagram, and let $\bw' \in W(\Omega)$ be a minimal lacing diagram from which we can obtain $\bw$ using a sequence of $K$-theoretic moves. Since $\bw'$ is minimal, by Corollary \ref{cor:OntoLaces} there is a $P \in \RPipes(v_0, v(\Omega))$ with $\pi(P)=\fc(\bw')$. In particular, $\fc(\bw') \in X_\Omega$. By Lemma \ref{lem:KTheoreticPermutations} and Lemma \ref{lem:KThEquivariance}, $\fc(\bw) \in X_{\Omega}$ as well since $\bw$ can be obtained from $\bw'$ by a sequence of $K$-theoretic moves. Since $\bw=LD(\fc(\bw))$, we have therefore proven that every $K$-theoretic lacing diagram can be realized as $LD(\bv)$ for some $\bv \in X_\Omega$. 
\end{proof}

\subsection{Proof of the bipartite type $A$ $K$-theoretic quiver component formula}
We are now ready to prove our main theorem. The proof is very similar to the proof of Theorem \ref{thm:3.9}.

\begin{theorem} \label{thm:4.16}
For any bipartite type A quiver locus $\Omega$, its K-polynomial can be written as
\begin{align*}
KQ_{\Omega}(\mathbf{t}/\mathbf{s})=\sum_{\mathbf{w} \in KW(\Omega)} (-1)^{|\mathbf{w}|-\codim (\Omega)} \mathfrak{G}_\mathbf{w}(\mathbf{t};\mathbf{s})\,.
\end{align*}
\end{theorem}

\begin{proof}
The $K$-theoretic pipe formula (Theorem \ref{thm:pipeformulas}) states that 
\begin{align*}
KQ_\Omega(\mathbf{t}/\mathbf{s})=\sum_{P \in \Pipes(v_0, v(\Omega))} (-1)^{|P\setminus P_\ast|-\codim(\Omega)} \left(1- \mathbf{t}/\mathbf{s}\right) ^{P\setminus P_*}.
\end{align*}
Noting that 
\begin{align*}
|P\setminus P_\ast| -\codim(\Omega) &= |P\setminus P_\ast| -( \ell(v(\Omega))- \ell(v_\ast)) =|P| - \ell(v(\Omega)), 
\end{align*}
we can rewrite this equation as 
\begin{align}\label{eqn:kpoly}
    KQ_\Omega(\mathbf{t}/\mathbf{s})=\sum_{P \in \Pipes(v_0, v(\Omega))} (-1)^{|P|-\ell(v(\Omega))} \left(1- \mathbf{t}/\mathbf{s}\right) ^{P\setminus P_*}.
\end{align}
By Proposition \ref{prop:pipeDreamBijection} and Corollary \ref{lem:partialperm}, Equation \eqref{eqn:kpoly} can be rewritten as
\begin{align*}
KQ_\Omega(\mathbf{t}/\mathbf{s})=\sum_{\mathbf{v} \in X_\Omega} \sum_{P \in \PipeNet(LD(\mathbf{v}))} (-1)^{|P|-\ell(v(\Omega))} (1-\mathbf{t}/\mathbf{s})^{P\setminus P_*}.
\end{align*}
Now, by Proposition \ref{prop:lacesBijection}, $X_\Omega$ is naturally in bijection with $KW(\Omega)$, and we may use Lemma \ref{lem:LDPerm} to write 
\begin{align*}
KQ_\Omega(\mathbf{t}/\mathbf{s})= \sum_{\mathbf{w} \in KW(\Omega)}  \sum_{P \in \PipeNet(\mathbf{w})} (-1)^{|P|-\ell(v(\Omega))} (1-\mathbf{t}/\mathbf{s})^{P\setminus P_*}.
\end{align*} 
It is therefore sufficient to prove the following equality for any fixed $\mathbf{w} \in KW(\Omega)$: 
\begin{align}\label{eqn:KTfactorization}
(-1)^{|\mathbf{w}| - \codim(\Omega)} \mathfrak{G}_\mathbf{w}(\mathbf{t};\mathbf{s})= \sum_{P \in \text{PipeNet}(\mathbf{w})} (-1)^{|P|-\ell(v(\Omega))}(1-\mathbf{t}/\mathbf{s})^{P \setminus P_*}.
\end{align}
To see this, we notice that
\begin{align*}
&\sum_{P \in \text{PipeNet}(\mathbf{w})} (-1)^{|P|-\ell(v(\Omega))}(1 - \mathbf{t}/\mathbf{s})^{P\setminus P_*} \\
=&(-1)^{- \ell(v(\Omega))+|P_\ast|}\sum_{P \in \text{PipeNet}(\mathbf{w})} \prod_{i=1}^n (-1)^{|P_i|} (1-\mathbf{t}^i/\mathbf{s}^i)^{P_i} \prod_{i=1}^n (-1)^{|P^i|} (1-\tilde{\mathbf{t}}^{i-1}/\tilde{\mathbf{s}}^i)^{\rot(P^i)} \\
\end{align*}
Using the fact that $-\ell(v(\Omega))+|P_\ast| = -\codim(\Omega)$ and factoring this polynomial as in the proof of Theorem \ref{thm:3.9}, we obtain
\begin{align*}
&(-1)^{-\codim(\Omega)} \left( \sum_{P_n \in \Pipes (w_n)} (-1)^{|P_n|} (1-\mathbf{t}^n/\mathbf{s}^n)^{P_n} \right) \\
&\times \sum_{[P] \in \text{PipeNet}(\mathbf{w}; \beta_n)} \prod_{i=1}^{n-1} (-1)^{|P_i|} (1-\mathbf{t}^i/\mathbf{s}^i)^{P_i} \prod_{i=1}^n (-1)^{|P^i|} (1-\tilde{\mathbf{t}}^{i-1}/\tilde{\mathbf{s}}^i)^{\rot(P^i)},
\end{align*}
where the sum over equivalence classes of pipe dreams is interpreted as in Remark \ref{rem:equivSums}. The term in parentheses can be rewritten as
\begin{align*}
      \sum_{P_n \in \Pipes (w_n)} (-1)^{|P_n|} (1-\mathbf{t}^n/\mathbf{s}^n)^{P_n} &=   \sum_{P_n \in \Pipes (w_n)} (-1)^{(|P_n| - |w_n|)+|w_n|} (1-\mathbf{t}^n/\mathbf{s}^n)^{P_n} \\
      &=  (-1)^{|w_n|}\sum_{P_n \in \Pipes (w_n)} (-1)^{|P_n|-\ell(c(w_n))} (1-\mathbf{t}^n/\mathbf{s}^n)^{P_n}.
\end{align*}
Noticing that
\begin{align*}
    \sum_{P_n \in \Pipes(w_n)}(-1)^{P_n - \ell(c(w_n)}(1- \bt^n/\bs^n)^{P_n} =\fG_{w_n}(\bt^n;\bs^n),
\end{align*}
we see that the right-hand-side of Equation \eqref{eqn:KTfactorization} can ultimately be simplified to 
\begin{align*}
&(-1)^{-\codim(\Omega)+|w_n|} \mathfrak{G}_{w_n}(\mathbf{t}^n;\mathbf{s}^n)& \\
\times &\sum_{[P] \in \PipeNet(\mathbf{w};\beta_n)} \prod_{i=1}^{n-1} (-1)^{|P_i|} (1-\mathbf{t}^i/\mathbf{s}^i)^{P_i} \prod_{i=1}^n (-1)^{|P^i|} (1-\tilde{\mathbf{t}}^{i-1}/\tilde{\mathbf{s}}^i)^{\rot(P^i)}. \\
\end{align*}
If we continue in this fashion, then after having factored each Grothendieck polynomial in the same way we obtain
\begin{align*}
(-1)^{-\codim(\Omega)+\sum_{i=1}^n|w_i|+|w^i|} \prod_{i=1}^n \mathfrak{G}_{w_i}(\mathbf{t}^i;\mathbf{s}^i) \prod_{i=1}^n \mathfrak{G}_{w^i}(\tilde{\mathbf t}^{i-1};\tilde{\mathbf s}^i) =(-1)^{|\mathbf{w}|-\codim(\Omega)} \mathfrak{G}_{\mathbf{w}}(\mathbf{t};\mathbf{s})\,,
\end{align*}
as desired.
\end{proof}

\bibliographystyle{amsalpha}
\bibliography{typeA}

\end{document}